\documentclass[11pt]{amsart}

\usepackage{amscd,amssymb,amsopn,amsmath,amsthm,mathrsfs,graphics,amsfonts,enumerate,verbatim,calc
}

\usepackage[OT2,OT1]{fontenc}
\newcommand\cyr{%
\renewcommand\rmdefault{wncyr}%
\renewcommand\sfdefault{wncyss}%
\renewcommand\encodingdefault{OT2}%
\normalfont
\selectfont}
\DeclareTextFontCommand{\textcyr}{\cyr} 

\usepackage{amssymb,amsmath}

\DeclareFontFamily{OT1}{rsfs}{}
\DeclareFontShape{OT1}{rsfs}{n}{it}{<-> rsfs10}{}
\DeclareMathAlphabet{\mathscr}{OT1}{rsfs}{n}{it}

\numberwithin{equation}{section}
\newtheorem{theorem}{Theorem}[section]
\newtheorem{lemma}[theorem]{Lemma}
\newtheorem{proposition}[theorem]{Proposition}
\newtheorem{corollary}[theorem]{Corollary}

\newtheorem{question}{Question}
\newtheorem{Problem}{Problem}

\newtheorem{maintheorema}{Main Theorem A}

\newtheorem{maintheoremb}{Main Theorem B}

\newtheorem{maintheoremc}{Main Theorem C}

\theoremstyle{definition}
\newtheorem{definition}[theorem]{Definition}
\newtheorem{remark}[theorem]{Remark}
\newtheorem{example}[theorem]{Example}

\theoremstyle{remark}

\newcommand{\Aut}{\operatorname{Aut}}
\newcommand{\Spec}{\operatorname{Spec}}

\newcommand{\Ht}{\operatorname{ht}}

\newcommand{\Hom}{\operatorname{Hom}}

\newcommand{\no}{\operatorname{no}}

\newcommand{\Frac}{\operatorname{Frac}}

\newcommand{\red}{\operatorname{red}}

\newcommand{\fm}{\frak{m}}

\newcommand{\fp}{\frak{p}}

\newcommand{\fa}{\frak{a}}

\newcommand{\Frob}{\operatorname{Frob}}

\newcommand{\cl}{\operatorname{cl}}

\begin{document}
\title[An embedding problem of Noetherian rings into the Witt vectors]
{An embedding problem of Noetherian rings into the Witt vectors}

\author[K.Shimomoto]{Kazuma Shimomoto}
\address{Department of Mathematics, College of Humanities and Sciences,
Nihon University, Setagaya-ku, Tokyo, 156-8550, Japan}
\email{shimomotokazuma@gmail.com}
\thanks{2000 {\em Mathematics Subject Classification\/}:13A35, 13B22, 13B35, 13B40, 13K05}

\keywords{Complete local ring, \'etale extension, Frobenius map, Witt vectors.}


\begin{abstract} 
The aim of this article is to prove some results on the existence of an integral extension domain of a complete local Noetherian domain in mixed characteristic $p>0$ having certain distinguished properties with respect to the Frobenius map. We prove the main results via Witt vectors and the method of maximal \'etale extensions. The main results are related to the existence problem of big Cohen-Macaulay algebras with certain distinguished properties via recent results of Y. Andr\'e.
\end{abstract}

\maketitle

\section{Introduction}

All rings in this article are assumed to be commutative and unitary. Let $A^+$ be the \textit{absolute integral closure} of an integral domain $A$ (see Definition \ref{absoluteintegral}). We set $\mathbb{F}_p:=\mathbb{Z}/p\mathbb{Z}$ for a prime integer $p>0$. The aim of this article is to investigate the following problem.

\begin{Problem}
\label{problem1}
Assume that $A$ is a $($not necessarily Noetherian$)$ integral domain of mixed characteristic $p>0$ and $A^+$ is the absolute integral closure of $A$. Then does there exist an $A$-algebra $T$ such that $A \subset T \subset A^+$ together with a non-zero non-unit element $\pi \in T$, $T/\pi T$ is an $\mathbb{F}_p$-algebra and the Frobenius endomorphism is bijective $($or surjective$)$ on $T/\pi T$ ?
\end{Problem}

In addition to the conditions stated in Problem \ref{problem1}, we require $A \to T$ to be a reasonably small integral extension (we will explain the meaning of "reasonably small" later). Problem \ref{problem1} has applications by combining it with the theory of (ramified) Witt vectors, especially to the construction of big Cohen-Macaulay algebras (see \cite{Ho07} and \cite{Rob12} and Proposition \ref{BigMac} below), and the construction of rings (called the \textit{Fontaine rings}) used in $p$-adic Hodge theory. One of the main streams into the $p$-adic Hodge theory is attributed to the \textit{Almost Purity Theorem} originally proved by Faltings (see \cite{DavKed}, \cite{KedRuo} and \cite{Sch} for the almost purity theorem and \cite{Shim2} for its application to the homological conjectures), In $p$-adic Hodge theory, the surjectivity of the Frobenius is an important issue. The element $\pi$ is usually taken to be a uniformizing parameter of a discrete valuation ring $V$. If the Frobenius endomorphism is bijective (resp. surjective) on an $\mathbb{F}_p$-algebra, then we say that it is \textit{perfect} (resp. \textit{semiperfect}). While almost nothing is known in the case of perfect algebras as required in Problem \ref{problem1}, the case of semiperfect algebras has been explored deeply in connection with \textit{Perfectoid Spaces} developed by Scholze \cite{Sch} and it has been strengthened as \textit{Perfectoid Abhyankar's Lemma} in Andr\'e's recent paper \cite{An1}. As a byproduct, he proved that a complete local domain of mixed characteristic admits a big Cohen-Macaulay algebra \cite{An2}. 

In this article, we attack Problem \ref{problem1} by taking $A$ to be a complete local Noetherian domain with mixed characteristic. We note that a key idea of Problem \ref{problem1} is contained in the following problem.

\begin{Problem}
\label{problem2}
Assume that $A$ is a $($not necessarily Noetherian$)$ normal domain of mixed characteristic $p>0$ and $A^+$ is the absolute integral closure of $A$. Then describe the maximal \'etale extension of $A$ inside $A^+$.
\end{Problem}

A precise definition of maximal \'etale extensions will be given in Definition \ref{maximaletale}. We now state our main theorems as a partial answer to Problem \ref{problem1} and Problem \ref{problem2}, which deals with the bijectivity of the Frobenius map (see Theorem \ref{theorem1}).

\begin{maintheorema}
Let $S$ be a complete local domain of mixed characteristic $p>0$ with finite residue field. Then there exists an $S$-algebra $T$ with a non-zero non-unit element $\pi \in T$ such that the following conditions hold:

\begin{enumerate}
\item[$\mathrm{(i)}$]
$T$ is a normal domain and $S \subset T \subset S^+$. 

\item[$\mathrm{(ii)}$]
$T/\pi T$ is a reduced $\mathbb{F}_p$-algebra.

\item[$\mathrm{(iii)}$]
For any prime ideal $P$ of $T$ that is minimal over $\pi T$, the Frobenius endomorphism is bijective on the quotient ring $T/P$.
\end{enumerate}
\end{maintheorema}

We make a couple of comments. The most crucial part of Main theorem A is that $\pi T$ is a radical ideal of $T$. Let $\{P_i\}_{i \in \Lambda}$ be the set of all prime ideals of $T$ that are minimal over $\pi T$. Then we have $\pi T=\sqrt{\pi T}=\cap_{i \in \Lambda} P_i$. If we ignore the requirement that $\pi T$ is a radical ideal, it is immediate to see that the absolute integral closure $T=S^+$ satisfies all other conditions (the surjectivity of the Frobenius map on $S^+/pS^+$ is clear in view of the fact that any monic polynomial $f(X) \in S^+[X]$ splits into the product of linear factors). The reason for assuming the finiteness of the residue field in the theorem is that it is crucial to use an extension of the Witt-Frobenius map on the (usual) $p$-typical Witt vectors to its ramifed extension. Historically, this idea was introduced by V. Drinfeld. We will see that our proof of the main theorem sheds light on the structure of $T$ in more details; the ring $T$ is relatively small compared with $S^+$ and carries specific information about the ramification over some big ring $R_{\infty}$, which has certain properties and is introduced in the main context (see Definition \ref{bigbasic}). Indeed, this is concerned about the second main theorem below. Let $\mathbf{W}(\mathbb{F})$ be the ring of Witt vectors of a finite field $\mathbb{F}$ of characteristic $p>0$ and let $(V,\pi,\mathbb{F})$ be a discrete valuation ring which is finite flat over $\mathbf{W}(\mathbb{F})$. Then we have the following result (see Corollary \ref{theorem2}):

\begin{maintheoremb}
Assume that $R:=V[[x_2,\ldots,x_d]] \hookrightarrow S$ is a module-finite extension of complete local domains such that $R[\frac{1}{a}] \to S[\frac{1}{a}]$ is \'etale for some $a \in R$ and the height of the ideal $(\pi,a)$ of $R$ is 2. Then the $S$-algebra $T$ in Main Theorem A can be taken to satisfy the following properties:
\begin{enumerate}
\item[$\mathrm{(i)}$]
There is a commutative diagram of integral domains
$$
\begin{CD}
R_{\infty} @>>> T \\
@AAA @AAA \\
R @>>> S \\
\end{CD}
$$
where each map is injective and integral. Moreover, the ring map
$$
R_{\infty}/\pi R_{\infty}[\frac{1}{a}] \to T/\pi T[\frac{1}{a}],
$$
which is induced by $R_{\infty} \to T$, is the filtered colimit of finite \'etale $R_{\infty}/\pi R_{\infty}[\frac{1}{a}]$-algebras and the Frobenius endomorphism is bijective on $T/\pi T[\frac{1}{a}]$.

\item[$\mathrm{(ii)}$]
Fix a prime ideal $P$ of $T$ that is minimal over $\pi T$. Then there exists a ring automorphism:
$$
\mathbf{F}:T \xrightarrow{\sim} T
$$
such that $\mathbf{F}(P)=P$ and the induced map $\overline{\mathbf{F}}:T/P \xrightarrow{\sim} T/P$ coincides with the $q$-th power map with $q:=|\mathbb{F}|$ and $\mathbb{F}=V/\pi V$.
\end{enumerate}
\end{maintheoremb}

The conclusion of the above theorem is that there is a lift of the $q$-th power map from $T/P$ to $T$. Combining both Main Theorem A and Main Theorem B, it seems reasonable to guess that $T/\pi T$ is a perfect algebra. However, we can only say that this is a subtle question. 

Next, let us turn our attention to the construction of a semiperfect algebra which allows a deep ramification over $p$, as a reasonably small integral extension over a complete local domain. In this case, we consider the situation where surjectivity of the Frobenius map holds, while injectivity of the Frobenius map fails. It should be noted that the resulting algebra is much smaller than its absolute integral closure. More precisely, we prove the following theorem (see Theorem \ref{theorem3}) which is valid for complete local domains with an \textit{arbitrary} residue field of characteristic $p>0$.

\begin{maintheoremc}
Let $S$ be a complete local domain with mixed characteristic $p>0$. Then there exists an $S$-algebra $T$ such that the following hold:

\begin{enumerate}
\item[$\mathrm{(i)}$]
$T$ is a normal domain and $S \subset T \subset S^+$.

\item[$\mathrm{(ii)}$]
There is an element $\pi \in T$ such that $\pi^p=p$ and the Frobenius endomorphism is surjective on $T/pT$, which induces an isomorphism:
$$
T/\pi T \cong T/pT.
$$

\item[$\mathrm{(iii)}$]
There exist a complete discrete valuation ring $V$, a regular local sub-algebra
$$
R:=V[[t_2,\ldots,t_d]] \subset T
$$
together with an element $a \in R$, and a complete local normal domain $S'$ such that $R \subset S' \subset T$, where $R \to S'$ is module-finite, $S' \to T$ is integral, the height of the ideal $(p,a)$ of $R$ is 2, and the localization maps:
$$
R[\frac{1}{a}] \to S'[\frac{1}{a}]~\mbox{and}~S'[\frac{1}{p}] \to T[\frac{1}{p}]
$$
are ind-\'etale. In particular, 
$$
R[\frac{1}{pa}] \to T[\frac{1}{pa}]
$$
is ind-\'etale.
\end{enumerate}
\end{maintheoremc}

Let us remark that the local domain $S$ in the above theorem has no inclusion relation with $R$. The present article is seen as a sequel of author's attempt \cite{Shim2} and \cite{Shim3} to understand the nature of rings of mixed characteristic via Frobenius and Witt vectors. Recently, Andr\'e proved that any complete local domain of mixed characteristic has a big Cohen-Macaulay algebra in \cite{An1}. We will apply the main results of this paper to study further properties of big Cohen-Macaulay algebras in \cite{Shim4}. The author believes that the results and techniques in this paper will shed more light on the structure of rings in mixed characteristic.

\subsection{Outline of the paper}
In \S~\ref{sec1}, we introduce some notation and recall definitions of some ring theory in the non-Noetherian context.

In \S~\ref{sec2}, we give only basic part of the theory of Witt vectors with emphasis on lifting of rings of positive characteristic with its Frobenius map to rings of mixed characteristic.

In \S~\ref{sec3}, we discuss normality and \'etale ring maps in the non-Noetherian context and it is important to give special care to how these notions are defined in this generality.

In \S~\ref{sec4}, we discuss the ramified Witt vectors due to Drinfeld to the extent we need (see \cite{CD} for more details) and its connection with the construction of big Cohen-Macaulay algebras. This section will be a key part for constructing some big rings.

In \S~\ref{sec5}, we discuss Gabber's refinement of classical Cohen's structure theorem on complete local rings.

In \S~\ref{sec6}, we introduce a basic ring denoted $R_{\infty}$ and discuss its basic properties. Then we introduce the notion of \textit{maximal \'etale extension} with respect to a torsion free ring extension $A \to B$ such that $A$ is normal domain and $B$ is reduced. The author thinks that this is a fundamental notion in commutative algebra, although no relevant reference for pure algebraists has been found (see also \cite{GroRay}).

In \S~\ref{sec7}, we introduce more rings of mixed characteristic. They are defined to be stable under the $q$-Witt-Frobenius map and these rings turn out to give a partial answer to the problems in the introduction.

In \S~\ref{sec8}, we establish Main Theorem A and Main Theorem B. Some remarks concerning rings constructed in this article are made for future's research.

In \S~\ref{sec9}, we establish Main Theorem C, where we construct a certain integral extension of a complete local domain which is semiperfect, but not perfect. Semiperfect algebras appear in the construction of Fontaine rings.

\section{Notation}
\label{sec1}

All rings in this article are commutative with unity. However, we do not always assume rings to be Noetherian. A \textit{local ring} is a Noetherian ring with a unique maximal ideal. The characteristic of this article is highly non-Noetherian and we will need to consider an infinite integral extension of a Noetherian domain.

\begin{definition}[\cite{Ar}]
\label{absoluteintegral}
Let $A$ be an integral domain. Then the \textit{absolute integral closure} of $A$ is defined to be the integral closure of $A$ in the algebraic closure of the field of fractions of $A$. We denote this ring by $A^+$.
\end{definition}

Note that if $A$ is not a field, then $A^+$ is not Noetherian. We refer the reader to \cite{As}, \cite{AsShim} for the homological aspect of the absolute integral closures of Noetherian domains.

\begin{definition}
Let $p>0$ be a prime number. Say that a ring $A$ is \textit{$p$-torsion free}, if $p$ is a non-zero divisor in $A$. Say that a $p$-torsion free ring $A$ is of \textit{mixed characteristic $p>0$}, if $pA \ne A$.
\end{definition}

\begin{definition}
An $\mathbb{F}_p$-algebra $B$ is \textit{perfect} (resp. \textit{semiperfect}), if the Frobenius map is bijective (resp. surjective) on $B$. 
\end{definition}

Let $A$ be an $\mathbb{F}_p$-domain. Then there are standard ways to find perfect algebras containing $A$. One is to take the perfect closure of $A$, which is obtained by adjoining all $p$-power roots of all elements of $A$. The second is to take the absolute integral closure of $A$. It is quite crucial to make an essential use of the Witt vectors in this article. There are many flavors of Witt vectors in the literature, however the only Witt vectors we consider are the ($p$-typical or ramified) Witt vectors of perfect $\mathbb{F}_p$-algebras. A basic reference of $p$-typical Witt vectors is Serre's book \cite{Se}. Another good source is an expository paper \cite{Rab}. The theory of the ramified Witt vectors is quite similar to the theory of the $p$-typical Witt vectors. However, it is necessary to assume the residue field $\mathbb{F}$ to be finite and the category of $\mathbf{W}(\mathbb{F})$-algebras is replaced with the category of algebras over a valuation ring that is finite over $\mathbf{W}(\mathbb{F})$ (see \cite{CD} for details). We will give a review of (ramified) Witt vectors to the extent we need in the main context.

The notation $\Frac(A)$ stands for the total ring of fractions for a ring $A$. If $A$ is reduced with only finitely many minimal primes, then $\Frac(A)$ is a finite direct product of fields. A ring map $R \to S$ is \textit{torsion free}, if every regular element in $R$ is also regular in $S$. This is equivalent to the condition that the natural map $S \to \Frac(R) \otimes_R S$ is injective.

\begin{remark}
One should be cautious in the use of the phrase: \textit{the Witt-Frobenius map on a ring $S$}. After a pair $(\mathbf{W}(A),i)$ has been fixed, where $i:S \hookrightarrow \mathbf{W}(A)$ is a ring injection, the Witt-Frobenius map on $S$ is defined by restricting it from $\mathbf{W}(A)$ to $S$. So it depends on an embedding of $S$ into some ring of Witt vectors.
\end{remark}

\section{$\pi$-adic deformation and Witt vectors}
\label{sec2}

\subsection{Witt vectors and the Frobenius}
In this section, we introduce a notion of $\pi$-adic deformation of rings and discuss its relation to Witt vectors. We will later discuss the ramified Witt vectors in detail. Fix a prime number $p>0$ and a ring $A$ of arbitrary characteristic. Denote by $\mathbf{W}(A)$ (resp. $\mathbf{W}_{p^n}(A)$) the ring of (\textit{p-typical}) \textit{Witt vectors} (resp. \textit{Witt vectors of length n}). Then one has a set-theoretic identity: $\mathbf{W}_{p^n}(A)=A^{n+1}$ and a ring-theoretic isomorphism: $\mathbf{W}(A) \cong \varprojlim_n \mathbf{W}_{p^n}(A)$. There is a well-defined ring homomorphism called the \textit{Witt-Frobenius map}: 
$$
\mathbf{F}:\mathbf{W}(A) \to \mathbf{W}(A),
$$ 
which is described as follows. If $A$ is a ring of prime characteristic $p>0$, then $\mathbf{F}(r_1,r_2,\ldots,r_n,\ldots)=(r_1^p,r_2^p,\ldots,r_n^p,\ldots)$. A more general formula is found in \cite[Remark 1.5]{DavKed}. We often use the symbol "Frob" to denote the $p$-th power map: $x \mapsto x^p$ for $x \in R$ and  ring $R$ of characteristic $p>0$ to distinguish it from the similar map $\mathbf{F}$ as above. It is easy to see that if $B$ is a perfect $\mathbb{F}_p$-domain which is not a field, then $B$ is not Noetherian. Fix a ring $B$ of prime characteristic $p>0$. Then can one find a ring $A$ of mixed characteristic $p>0$ such that $A/pA \cong B$? This leads to the following definition:

\begin{definition}
Let $B$ be a ring. Then a ring $A$ with a non-zero non-unit element $\pi \in A$ is a \textit{$\pi$-adic deformation} of $B$, if $A$ is $\pi$-torsion free, complete in the $\pi$-adic topology and $A/\pi A \cong B$.
\end{definition}

Let us collect some basic properties of Witt vectors which we often use (see \cite{Se} for more details).

\begin{enumerate}
\item[$\bullet$]
If $A$ is a perfect $\mathbb{F}_p$-algebra, then $\mathbf{W}(A)$ is a $p$-adically complete, $p$-torsion free algebra, and there is a surjection $\pi_A:\mathbf{W}(A) \twoheadrightarrow A$ with kernel generated by $p$.

\item[$\bullet$]
If $A$ is a perfect $\mathbb{F}_p$-algebra, then there is a multiplicative injective map $f_A:A \to \mathbf{W}(A)$ such that the composite map
$$
A \xrightarrow{f_A} \mathbf{W}(A) \xrightarrow{\pi_A} A
$$
is an identity map. $f_A$ is called the \textit{Teichm\"uller mapping} (see Proposition \ref{Serre} below). For any given $x \in \mathbf{W}(A)$, there exists a unique sequence of elements $a_0,a_1,a_2,\ldots \in A$ such that $x=f_A(a_0)+f_A(a_1) p+f_A(a_2) p^2+\cdots$, which we call the \textit{Witt representation} of $x$. We note the following simple fact.
$$
p|x \iff a_0=0.
$$

\item[$\bullet$]
If $A \to B$ is a ring homomorphism of perfect $\mathbb{F}_p$-algebras, then there is a unique ring homomorphism $\mathbf{W}(A) \to \mathbf{W}(B)$
making the following commutative square:
$$
\begin{CD}
\mathbf{W}(A) @>>> \mathbf{W}(B) \\
@V\pi_AVV @V\pi_BVV \\
A @>>> B \\
\end{CD}
$$
\end{enumerate}

Indeed, we have a unique $p$-adic deformation for a perfect $\mathbb{F}_p$-algebra, as stated in the following proposition.

\begin{proposition}
\label{prop}
Let $A$ be a perfect $\mathbb{F}_p$-algebra. Then $A$ admits a unique $p$-adic deformation $\mathcal{A}$. Moreover, if $A \to B$ is a ring homomorphism of perfect $\mathbb{F}_p$-algebras, there exists a unique commutative diagram:
$$
\begin{CD}
\mathcal{A} @>>> \mathcal{B} \\
@VVV @VVV \\
A @>>> B
\end{CD}
$$
such that both $\mathcal{A} \to A$ and $\mathcal{B} \to B$ are $p$-adic deformations. In particular, we have
$$
\Aut(A) \cong \Aut(\mathcal{A}).
$$
\end{proposition}

\begin{proof}
The ring $\mathcal{A}$ is given as the Witt vectors $\mathbf{W}(A)$ and the proof of the uniqueness is found in \cite[Proposition 10 in Chapter II \S 5]{Se}. Alternatively, one may use the cotangent complex to avoid the use of Witt vectors \cite[Theorem 5.11 and Theorem 5.12]{Sch}.
\end{proof}

\begin{lemma}
\label{p-adic}
Assume that $B$ is a reduced ring. Then a $\pi$-adic deformation of $B$ is also  reduced.
\end{lemma}

\begin{proof}
Let $A$ be a $\pi$-adic deformation of $B$. Since $A$ is $\pi$-adically complete, it is $\pi$-adically separated. Assume that there is an element $y_1 \in A$ such that $y_1^n=0$ for some $n>0$. Then since $B=A/\pi A$ is reduced, the image of $y_1$ in $B$ is zero. Thus, we have $y_1=\pi y_2$ for some $y_2 \in A$ and $\pi^n y_2^n=0$. Since $\pi$ is a non-zero divisor, we have $y_2^n=0$. Then we may find $y_3 \in A$ for which $y_2=\pi y_3$ and $y_1=\pi^2 y_3$. Continuing this argument, we get 
$$
y_1 \in \bigcap_{n>0} \pi^n A=0,
$$
due to the fact that $A$ is $\pi$-adically separated. Hence $y_1=0$, as desired.
\end{proof}

\section{Normality of rings}
\label{sec3}

\subsection{\'Etale ring map}
Let us start with a review on \'etale ring maps in a general form. We refer the reader to \cite{Stacks} as a standard reference.

\begin{definition}
Let $A \to B$ be a ring map. Then $A \to B$ is \textit{\'etale} if it is flat, unramified and of finite presentation. $A \to B$ is \textit{finite \'etale} if it is \'etale and integral. Finally, $A \to B$ is \textit{ind-\'etale} if $B$ is obtained as the filtered colimit of \'etale $A$-algebras.
\end{definition}

For \'etale ring maps, we have standard results.

\begin{lemma}
Let $A \to B$ be a ring map.

\begin{enumerate}
\item[$\mathrm{(i)}$]
A composite of \'etale ring maps is \'etale.

\item[$\mathrm{(ii)}$]
The base change of an \'etale ring map is \'etale.

\item[$\mathrm{(iii)}$]
Assume that $A \to B$ is of finite presentation. Then $A \to B$ is formally \'etale if and only if it is \'etale.
\end{enumerate}
\end{lemma}

\begin{example}
\begin{enumerate}
\item
Let $A$ be a ring. A simple example is that for a non-nilpotent element $f \in A$, the localization $A \to A[\frac{1}{f}]$ is \'etale. Indeed, it is clear that it is flat and unramified. Furthermore, we see by the universal property of localization that $A[\frac{1}{f}]$ is isomorphic to a finite presentation $A[x]/(fx-1)$. The most important case is when $f$ is an idempotent. 

\item
There is an example of a ring map $A \to B$ which is finite, flat ,but not of finite presentation. So it is necessary to assume all the specified conditions in the definition of \'etale ring maps.
\end{enumerate}
\end{example}

\subsection{Normality criteria}
By a \textit{normal ring}, we mean a ring $A$ such that $A_{\fp}$ is an integrally closed domain in its field of fractions for all prime ideals $\fp$ of $A$. In particular, $A$ is reduced.

\begin{lemma}
\label{normal}
Assume that $A$ is a reduced ring with finitely many minimal prime ideals.
\begin{enumerate}
\item[$\mathrm{(i)}$]
$A$ is normal if and only if it is integrally closed in its total ring of fractions. 

\item[$\mathrm{(ii)}$]
If $A$ is a normal ring with minimal prime ideals $\fp_1,\ldots,\fp_n$, then 
$$
A \cong A/\fp_1 \times \cdots \times A/\fp_n.
$$
\end{enumerate}
\end{lemma}

\begin{proof}
The first statement is in \cite[Lemma 2.1.15]{SwHu} and the second statement is in \cite[Corollary 2.1.13]{SwHu}.
\end{proof}

In case that $A$ has infinitely many minimal prime ideals, we have the following lemma.

\begin{lemma}
\label{normallemma}
\begin{enumerate}
\item[$\mathrm{(i)}$]
If $A$ is a normal ring, then it is integrally closed in its total ring of fractions.

\item[$\mathrm{(ii)}$]
Assume that $A \subset B$ is a torsion free ring extension such that $B$ is integrally closed in the total ring of fractions of $B$. Then the integral closure of $A$ in the total ring of fractions of $A$ is contained in $B$.
\end{enumerate}
\end{lemma}

\begin{proof}
(i): Let $\Frac(A)$ be the total ring of fractions and let $x \in \Frac(A)$ be integral over $A$. Define an ideal of $R$:
$$
I:=\{b \in A~|~bx \in A\}.
$$
For a prime ideal $\fp$ of $A$, the natural map $A_{\fp} \to \Frac(A) \otimes_A A_{\fp}$ is injective. Since $A_{\fp}$ is a domain, we may naturally view $\Frac(A) \otimes_A A_{\fp}$ as a sub-algebra of $\Frac(A_{\fp})$. Moreover, $A_{\fp}$ is an integrally closed domain, so we have $x \otimes 1 \in A_{\fp}$ and there exist elements $f \in A \setminus \fp$ and $a \in A$ such that $x \otimes 1=a \otimes \frac{1}{f}$. Writing this as $(fx-a) \otimes 1=0$, we may find $g \in A \setminus \fp$ such that $gfx=ga.$ Since $ga \in A$, we must get $gf \in I$ and $gf \in A \setminus \fp$. Since $\fp$ range over all prime ideals, the ideal $I$ cannot be contained in any prime ideal. Hence $I=A$ and $x \in A$.

(ii): This follows from the fact that the injection $A \subset B$ extends to an injection $\Frac(A) \subset \Frac(B)$ and the following commutative diagram:
$$
\begin{CD}
\Frac(A) @>>> \Frac(B) \\
@AAA @AAA \\
A @>>> B \\
\end{CD}
$$
\end{proof}

It is noted that the converse of (i) in Lemma \ref{normallemma} is not true.

\begin{lemma}
\label{Koszul}
Let $(f,g)$ be a regular sequence $($in this order$)$ in an integral domain $A$. Then
$$
A=A[\frac{1}{f}] \cap A[\frac{1}{g}].
$$
Moreover, if both $A[\frac{1}{f}]$ and $A[\frac{1}{g}]$ are integrally closed domains, then $A$ is integrally closed.
\end{lemma}

\begin{proof}
For a given $x \in A[\frac{1}{f}] \cap A[\frac{1}{g}]$, let $m$ be the smallest non-negative integer such that $f^m x \in A$ and put $a=f^m x$. We take another presentation $x=\frac{b}{g^n}$ with $b \in A$, where $n$ is a positive integer, but not necessarily the smallest possible choice. So we have $x=\frac{a}{f^m}=\frac{b}{g^n}$. Then it suffices to prove that $m=0$. Assume $m>0$ to prove the lemma by contradiction. We have $bf^m=ag^n$. Since $(f,g)$ is a regular sequence, there exists an element $a' \in A$ such that $a=a'f$. But then we must have $x=\frac{a'f}{f^m}=\frac{a'}{f^{m-1}}$, which contradicts the minimality of $m$. Hence $m=0$, showing that $x \in A$. Since the intersection of normal domains is normal, the second assertion is clear.
\end{proof}

\begin{lemma}
\label{cor}
Assume that $A$ is a $($not necessarily Noetherian$)$ integral domain and $\pi A$ is a non-zero principal prime ideal of $A$. Let $A_{(\pi)}$ denote the localization of $A$ with respect to the multiplicative set $A \setminus \pi A$ and $B_{(\pi)}:=B \otimes_A A_{(\pi)}$. Then we have the following statements.

\begin{enumerate}
\item[$\mathrm{(i)}$]
If $A_{(\pi)}$ is a normal domain, then $A$ is integrally closed in $A[\frac{1}{\pi}]$.

\item[$\mathrm{(ii)}$]
Let $A \hookrightarrow B$ be an integral extension of normal domains. If $B_{(\pi)}/\pi B_{(\pi)}$ is a reduced ring, then $B/\pi B$ is reduced. In particular, this is satisfied if $A_{(\pi)} \to B_{(\pi)}$ is an ind-\'etale map.
\end{enumerate}
\end{lemma}

\begin{proof}
(i): Let $A^{\no}$ be the normalization of $A$ in $\Frac(A)$. Then it suffices to show that $A=A[\frac{1}{\pi}] \cap A^{\no}$. Take an element $x=\frac{a}{\pi^n} \in A[\frac{1}{\pi}] \cap A^{\no}$ for $n>0$. Then since $A_{(\pi)}$ is a normal domain, we have $A_{(\pi)}=(A^{\no})_{(\pi)}$. Thus we can write $x=\frac{b}{y}$ for some $y \in A \setminus \pi A$ and $b \in A$. Since $A/\pi A$ is an integral domain by assumption, $(\pi,y)$ is an $A$-regular sequence and we get
$$
x=\frac{a}{\pi^n}=\frac{b}{y} \in A[\frac{1}{\pi}] \cap A[\frac{1}{y}]=A
$$
in view of Lemma \ref{Koszul}.

(ii): Assume that $\pi B$ is not radical and deduce a contradiction. Then there exists $t \in B$ such that $t^N \in \pi B$ and $t \notin \pi B$ for some $N>0$. But since $\pi B_{(\pi)}$ is radical by assumption, we have $t \in \pi B_{(\pi)}$ and so we can write
$$
t=\frac{\pi b}{x}~\mbox{for some}~x \in A \setminus \pi A~\mbox{and}~b \in B.
$$
Since $t \in B$, we have $\big(\frac{b}{x}\big)=\big(\frac{t}{\pi}\big) \in B[\frac{1}{\pi}]$ and $\big(\frac{b}{x}\big) \in B[\frac{1}{x}]$. Since $A/\pi A$ is an integral domain, $x$ is a regular element on $A/\pi A$. Hence, $(\pi,x)$ is a regular sequence on $A$ and we have $A=A[\frac{1}{\pi}] \cap A[\frac{1}{x}]$ by Lemma \ref{Koszul}. On the other hand, we claim that
$$
A=A[\frac{1}{\pi}] \cap A[\frac{1}{x}] \to B[\frac{1}{\pi}] \cap B[\frac{1}{x}]
$$
is an integral extension. For this, let $\Frac(A)[X]$ be a polynomial ring over $\Frac(A)$ and let $F(X) \in \Frac(A)[X]$ be the monic minimal polynomial of an element $\alpha \in B[\frac{1}{\pi}] \cap B[\frac{1}{x}]$. Then $\alpha$ is integral over $A[\frac{1}{\pi}]$ and $A[\frac{1}{x}]$ which are both normal domains by normality of $A$. It follows from \cite[Theorem 2.1.17]{SwHu} that $F(X) \in A[\frac{1}{\pi}][X]$ and $F(X) \in A[\frac{1}{x}][X]$. Thus, we have that $F(X) \in A[X]$ and $A \to B[\frac{1}{\pi}] \cap B[\frac{1}{x}]$ is an integral extension of integral domains. Since $B[\frac{1}{\pi}]$ and $B[\frac{1}{x}]$ are normal domains, $B[\frac{1}{\pi}] \cap B[\frac{1}{x}]$ is also normal. This implies that $B[\frac{1}{\pi}] \cap B[\frac{1}{x}]=B$ by normality of $B$. Finally,
$$
\frac{b}{x} \in B[\frac{1}{\pi}] \cap B[\frac{1}{x}]=B,
$$
which yields a contradiction to our hypothesis. Hence we must have $t \in \pi B$.
\end{proof}

\begin{remark}
\label{Frobeniusetale}
Assume that $A \to B$ is a \'etale ring map. If $A/pA$ is a perfect $\mathbb{F}_p$-algebra, then $B/pB$ is also a perfect $\mathbb{F}_p$-algebra. This fact holds more generally for an \textit{absolutely flat} map $A \to B$ (see \cite{Oli} for its definition and \cite[Theorem 3.5.13]{GR} or \cite[Lemma 3.1.5]{KedRuo} for the proof).
\end{remark}

\section{Ramified Witt vectors}
\label{sec4}

\subsection{Ramified Witt vectors}
In this section, we discuss the ramified version of the $p$-typical Witt vectors, which was introduced by Drinfeld \cite{Dr}. We will combine the ramified Witt vectors with Cohen-Gabber theorem, which will be discussed in the next section. Traditionally, the ramified Witt vectors are defined by taking Witt-type polynomials with respect to $\pi$, where $\pi$ is a fixed parameter of a complete discrete valuation ring. It is instructive for us to avoid the heavy calculus of Witt vectors, so I will take a more direct approach to define the ramified Witt vectors whose exposition is found in \cite{Bor} and \cite{CD}. A finite flat extension of discrete valuation rings $(V,\pi_V) \to (W,\pi_W)$ is \textit{totally ramified}, if one has $\pi_V W=\pi_W^f W$ with $f=[W:V]$. Throughout, we fix a perfect field $k$ of characteristic $p>0$ and a totally ramified extension of discrete valuation rings $\mathbf{W}(k) \to V$ together with an element $\pi$ of $V$ generating the maximal ideal. Then it is known that $V=\mathbf{W}(k)[\pi]$ and $\pi$ is defined by an Eisenstein polynomial. To define the $q$-Witt-Frobenius map on the ring of ramified Witt vectors, we need to impose the following restriction on $(V,\pi_V)$.
\begin{enumerate}
\item[($\bf{Fin}$):]
The residue field of $V$ is finite and consists of $q=p^e$ elements. Henceforth, we denote the finite residue field by $\mathbb{F}_q$ or simply by $\mathbb{F}$.
\end{enumerate}

This restriction lets us consider the category of complete local Noetherian rings with finite residue field. In general, there is no extension of the (iterated) Witt-Frobenius map from $\mathbf{W}(k)$ to its finite extension, unless it is an \'etale extension or $k$ is a finite field. First, let $\mathbf{W}(k) \to A$ be a finite \'etale extension. Then one has $A \cong \mathbf{W}(k')$ for a finite extension $k \to k'$ and $A$ has the Witt-Frobenius map. Next, let $\mathbf{W}(k) \to A$ be a (possibly ramified) finite extension and $k$ is a finite field with $|k|=p^e$. Then the $e$-th iterated Frobenius map is the identity on $\mathbf{W}(k)$ and hence it extends as the identity on $A$.

\begin{definition}
Let $(V,\pi)$ satisfy the condition ($\bf{Fin}$) and let $A$ be a perfect $V/\pi V$-algebra. Then the \textit{ring of ramified Witt vectors} of $A$ with respect to $(V,\pi)$ is defined to be the tensor product:
$$
\mathbf{W}(A) \otimes_{\mathbf{W}(\mathbb{F}_q)} V,
$$
which we denote by $\mathbf{W}_{\pi}(A)$.
\end{definition}

It defines a functor from the category of perfect $V/\pi V$-algebras to the category of $V$-algebras. The above definition is easy to work with to prove basic results on the ramified Witt vectors. The following proposition allows one to expand an element of the ramified Witt vectors as the $\pi$-adic representation as in the $p$-typical Witt vectors.

\begin{proposition}
\label{Serre}
Let $A$ be a ring equipped with a decreasing sequence of ideals $\fa_1 \supset \fa_2 \supset \cdots$ with the property that $\fa_n \cdot \fa_m \subset \fa_{m+n}$ and $A$ is complete and separated with respect to the topology defined by this sequence. Assume further that $K:=A/\fa_1$ is a perfect $\mathbb{F}_p$-algebra. Then there exists one and only one system of representatives $($also called a Teichm\"uller mapping$)$
$$
f_A:K \to A
$$
such that $f_A$ is multiplicative and the image $f_A(K) \subset A$ is the set of all elements of $A$ which admit $p^n$-th roots for all $n>0$. In particular, assume that $\fa_n=\pi^nA$ for all $n>0$ and a non-zero divisor $\pi \in A$. Then any element $x \in A$ admits the following presentation:
$$
x=\sum_{i=0}^{\infty}f_A(a_i)\pi^i
$$
in which the sequence $a_0,a_2,\ldots$ is uniquely determined.
\end{proposition}

\begin{proof}
The first part is found in \cite[Proposition 8 in Chapter II \S 4]{Se}. For the second part, fix an element $x \in A$. Then we may find $a_0 \in K$ such that $x-f_A(a_0) \equiv \pi A$. Then $x=f_A(a_0)+\pi x_1$ with $x_1 \in A$ and apply the same process to $x_1$. Eventually, since $A$ is $\pi$-adically complete and separated, the presentation
\begin{equation}
\label{series}
\sum_{i=0}^{\infty}f_A(a_i)\pi^i
\end{equation}
converges to $x$. And it is easy to see that the sequence $a_0,a_2,\ldots$ is uniquely determined. Conversely, an element of the form $(\ref{series})$ converges in $A$. 
\end{proof}

The next proposition is an extension of Proposition \ref{prop}, which simply says that the ring of ramified Witt-vectors gives a unique $\pi$-adic deformation of a perfect $V/\pi V$-algebra.

\begin{proposition}
\label{Serre2}
Let $A$ be a perfect $V/\pi V$-algebra. Then the ramified Witt vectors $\mathbf{W}_{\pi}(A)$ is the unique $\pi$-adically complete and separated $\pi$-torsion free $V$-algebra with an isomorphism 
$$
\mathbf{W}_{\pi}(A)/\pi \mathbf{W}_{\pi}(A) \cong A.
$$
\end{proposition}

\begin{proof}
This is found in \cite[Proposition 2.11]{CD}.
\end{proof}

The next proposition describes the detailed structure of ramified Witt vectors.

\begin{proposition}
\label{almostCM}
With notation just as above, let $A$ be a perfect $V/\pi V$-algebra. Then we have the following assertions:

\begin{enumerate}
\item[$\mathrm{(i)}$]
The $V$-algebra $\mathbf{W}_{\pi}(A)$ is a finite free $\mathbf{W}(A)$-module whose free basis is given by $1,\pi,\ldots,\pi^{f-1}$ with $f=[V:\mathbf{W}(\mathbb{F}_q)]$.

\item[$\mathrm{(ii)}$]
There is a unique $V$-algebra isomorphism:
$$
\mathbf{F}_{\pi}:\mathbf{W}_{\pi}(A) \to \mathbf{W}_{\pi}(A)
$$
such that the induced map $\overline{\mathbf{F}}_{\pi}:\mathbf{W}_{\pi}(A)/\pi \mathbf{W}_{\pi}(A) \to \mathbf{W}_{\pi}(A)/\pi \mathbf{W}_{\pi}(A)$ coincides with the $e$-th iterated Frobenius automorphism.

\item[$\mathrm{(iii)}$]
If $A$ is a completely integrally closed domain, then $\mathbf{W}_{\pi}(A)$ is an integrally closed domain.
\end{enumerate}
\end{proposition}

\begin{proof}
(i): This is quite evident from the construction.

(ii): We may define the map $\mathbf{F}_{\pi}$ by setting
$$
\mathbf{F}_{\pi}(r \otimes s):=\mathbf{F}^e(r) \otimes s
$$
for $r \in \mathbf{W}(A)$ and $s \in V$. This is well-defined since $\mathbf{F}^e$ acts trivially on $\mathbf{W}(\mathbb{F}_q)$ (this is exactly where we need that the residue field is finite). Alternately, we may apply Proposition \ref{Serre}. Namely, for an element $x=\sum_{i=0}^{\infty}f_A(a_i)\pi^i \in \mathbf{W}_{\pi}(A)$, we define
$$
\mathbf{F}_{\pi}(x)=\sum_{i=0}^{\infty}f_A(\Frob^e(a_i))\pi^i.
$$

(iii): Under the stated assumption, $\mathbf{W}(A)$ is an integrally closed domain by \cite[Corollary 5.9]{Shim3}. Let $\mathbf{W}(A)_{(p)}$ be the localization of $\mathbf{W}(A)$ at the prime ideal $(p)$. Then by \cite[Proposition 5.2]{Shim3}, $\mathbf{W}(A)_{(p)}$ is an unramified discrete valuation ring. Since $V$ is totally ramified over $\mathbf{W}(\mathbb{F}_q)$, it follows that $\Frac(\mathbf{W}(A))$ and $\Frac(V)$ are linearly disjoint over $\Frac(\mathbf{W}(\mathbb{F}_q))$, in which all relevant rings are considered as sub-algebras of $\Frac(\mathbf{W}(A)^+)$, where $\mathbf{W}(A)^+$ is the absolute integral closure of $\mathbf{W}(A)$. Since $\mathbf{W}(\mathbb{F}_q) \to V$ is flat, there is an injection:
$$
\mathbf{W}(A) \otimes_{\mathbf{W}(\mathbb{F}_q)}V \hookrightarrow \Frac(\mathbf{W}(A)) \otimes_{\mathbf{W}(\mathbb{F}_q)}V,
$$
and the target ring is a field by linearly disjoint property. Then there is an isomorphism of domains: $\mathbf{W}_{\pi}(A)=\mathbf{W}(A) \otimes_{\mathbf{W}(\mathbb{F}_q)} V \cong \mathbf{W}(A)[\pi]$. Since 
$$
\mathbf{W}(A)[\frac{1}{p}] \to \mathbf{W}(A)[\pi][\frac{1}{p}]
$$ 
is finite \'etale, it follows that $\mathbf{W}_{\pi}(A)[\frac{1}{p}]=\mathbf{W}_{\pi}(A)[\frac{1}{\pi}]$ is an integrally closed domain. On the other hand, $(\pi)$ is a principal prime ideal of $\mathbf{W}_{\pi}(A)$ and the localization $\mathbf{W}_{\pi}(A)_{(\pi)}$ is a discrete valuation ring. Applying Lemma \ref{cor}, the integral domain $\mathbf{W}_{\pi}(A)$ is integrally closed in $\mathbf{W}_{\pi}(A)[\frac{1}{\pi}]$. Hence $\mathbf{W}_{\pi}(A)$ is an integrally closed domain.
\end{proof}

We will call $\mathbf{F}_{\pi}$ the \textit{q-Witt-Frobenius map}. We will use the following lemma to embed a complete local domain into the ring of ramified Witt vectors.

\begin{lemma}
\label{embed}
Let $A$ be an algebra which is flat over $(V,\pi)$ and assume that $A/\pi A$ is a perfect $V/\pi V$-algebra and $A$ is $\pi$-adically separated. Then there is a canonical injection:
$$
A \hookrightarrow \mathbf{W}_{\pi}(A/\pi A),
$$
which lifts the identity map on $A/\pi A$ and which sends $\pi \in A$ to $\pi \in \mathbf{W}_{\pi}(A/\pi A)$.
\end{lemma}

\begin{proof}
By \cite[Lemma 4.2]{Shim3}, the $\pi$-adic completion of $A$, denoted by $A^{\wedge}$, will be a $\pi$-torsion free, $\pi$-adically complete ring such that $A^{\wedge}/\pi A^{\wedge}$ is a perfect $V/\pi V$-algebra. Hence we have a canonical identification:
$$
A^{\wedge}=\mathbf{W}_{\pi}(A/\pi A)
$$
into which $A$ is embedded by the assumption that $A$ is $\pi$-adically separated. Hence we get the canonical injection.
\end{proof}

\subsection{Big Cohen-Macaulay algebra}
We give a simple application of the ramified Witt vectors to the construction of big Cohen-Macaulay algebras. Let $(R,\fm)$ be a local Noetherian ring. An $R$-algebra $T$ is a \textit{big Cohen-Macaulay algebra}, if there is a system of parameters of $R$ that is a regular sequence on $T$ and $\fm T \ne T$. Recall that $A^+$ is a big Cohen-Macaulay algebra for a complete local domain $A$ of characteristic $p>0$. In other words, every system of parameters of $A$ is a regular sequence on $A^+$. This is the main result of \cite{HH92}. The following proposition does not seem to be found in any research article.

\begin{proposition}
\label{BigMac}
Assume that $(S,\fm)$ is a complete local domain which is flat over $(V,\pi)$. Assume that there is a sub-algebra $S \subseteq T \subseteq S^+$ such that $T/\pi T$ is a perfect $\mathbb{F}_p$-algebra. Then $S$ maps to a big Cohen-Macaulay algebra.
\end{proposition}

\begin{proof}
Let $T^{\wedge}$ denote the $\pi$-adic completion of $T$. Then $T^{\wedge}$ is $\pi$-torsion free in view of \cite[Lemma 4.2]{Shim3} and it is the unique $\pi$-adic deformation of the perfect $\mathbb{F}_p$-algebra $T/\pi T$. Hence $\mathbf{W}_{\pi}(T/\pi T) \cong T^{\wedge}$. Note that $S/\pi S$ is a complete local ring of characteristic $p>0$ and the induced map $S/\pi S \to T/\pi T$ is an integral extension and $T/\pi T$ maps to a balanced big Cohen-Macaulay perfect $\mathbb{F}_p$-algebra which is constructed as follows: Choose a prime ideal $\mathcal{P}$ of $T$ that is minimal over $\pi T$. Now every system of parameters of $S/\pi S$ is a regular sequence on the absolute integral closure $(T/\mathcal{P})^+$ of $T/\mathcal{P}$ by \cite{HH92} and the composite ring map
$$
S \to T^{\wedge} \cong \mathbf{W}_{\pi}(T/\pi T) \to \mathbf{W}_{\pi}((T/\mathcal{P})^+)
$$
gives a big Cohen-Macaulay $S$-algebra.
\end{proof}

In general, if $T$ is an algebra with $S \subset T \subset S^+$, the principal ideal $\pi T$ is easily checked to be radical in certain cases, while the surjectivity of the Frobenius on $T/\pi T$ is quite subtle.

\section{Cohen-Gabber theorem}
\label{sec5}

Gabber recently proved an improved version of Cohen's structure theorem for complete local rings. One aspect is that it gives a control of ramification of a module-finite extension from a complete regular local ring. The original statement is in a more general form, but we state it in a form that suffices for our purpose. We discuss only the mixed characteristic case (see \cite[Theorem 3.5]{Illusie} or \cite[Th\'eor\`eme 4.2.2 at page 65]{Book} for a precise statement and its proof). For a field $k$ of characteristic $p>0$, let $\mathrm{C}(k)$ denote a unique unramifed complete discrete valuation ring with residue field $k$.

\begin{theorem}[Gabber]
\label{gabber}
Assume that $(A,\fm)$ is a complete local normal domain of dimension $d \ge 2$ and mixed characteristic $p>0$, with residue field $k$. Then there exists a torsion free module-finite extension $A \hookrightarrow B$ such that $B$ is a normal domain with residue field $k'$, together with a totally ramified extension of complete discrete valuation rings $\mathrm{C}(k') \to V$, a system of parameters $(p,t_2,\ldots,t_d)$ of $B$, and a torsion free module-finite extension
$$
V[[t_2,\ldots,t_d]] \to B,
$$
such that $V[[t_2,\ldots,t_d]][\frac{1}{a}] \to B[\frac{1}{a}]$ is \'etale for some $a \in V[[t_2,\ldots,t_d]] \setminus \pi V[[t_2,\ldots,t_d]]$, where $\pi$ is a unifomizing parameter $\pi$ of $V$. 
\end{theorem}

The residue field $k'$ of $V$ is, in general, a non-trivial extension of $k$.

\begin{example}
Let $R:=\mathbf{W}(k)[[x_2,\ldots,x_d]] \to S$ be a module-finite extension of normal domains such that $R[\frac{1}{pa}] \to S[\frac{1}{pa}]$ is \'etale with $\Ht(p,a)=2$. Consider the following question:
\begin{enumerate}
\item[$\bullet$]
Is there a factorization $R \xrightarrow{i} T \xrightarrow{j} S$ such that $T$ is normal and both $i \otimes R[\frac{1}{a}]$ (resp. $i \otimes R[\frac{1}{p}]$) and $j \otimes R[\frac{1}{p}]$ (resp. $j \otimes R[\frac{1}{a}]$) are \'etale? 
\end{enumerate}
However, this is not always the case and the following example illustrates a reason of the utility of Cohen-Gabber theorem in the proof of our main theorems. Fix an algebraically closed field $k$ of characteristic $p>2$ together with a module-finite extension of normal domains:
$$
R=\mathbf{W}(k)[[x]] \to S=\mathbf{W}(k)[[x]][t]/(t^2-px).
$$
By an easy calculation, $R[\frac{1}{px}] \to S[\frac{1}{px}]$ is an \'etale extension of degree 2, and there is no proper intermediate normal ring between $R$ and $S$. Let $T$ be the normalization of $(S \otimes_{\mathbf{W}(k)}V)_{\red}$, where $V=\mathbf{W}(k)[s]/(s^2-p)$ is a ramified discrete valuation ring with a maximal ideal $\pi V$. We claim that $T/\pi T$ is reduced. Localizing the module-finite map $V[[x]] \to T$ at the height-one prime $(\pi)$ of $V[[x]]$, we get a module-finite map
$$
V[[x]]_{(\pi)} \to T_{(\pi)}
$$
and $\pi T_{(\pi)}$ is radical. Thus, $\pi T$ is a radical ideal in view of Lemma \ref{cor}.
\end{example}

\section{Basic set-up}
\label{sec6}

\subsection{Basic ring}
We fix some notation. Let $\mathbb{F}_q$ be a \textit{finite} field consisting of $q=p^e$ elements and denote by $\mathbf{W}(\mathbb{F}_q)$ the ring of Witt vectors and let $(V,\pi,\mathbb{F}_q)$ be a complete discrete valuation ring of mixed characteristic $p>0$. That is, we continue to assume ($\bf{Fin}$) throughout this section. Assume that $R=V[[x_2,\ldots,x_d]] \to S$ is a module-finite extension of local normal domains such that 
$$
R_{(\pi)} \to S_{(\pi)}:=S \otimes_R R_{(\pi)}
$$ 
is finite \'etale, where $R_{(\pi)}$ is the discrete valuation ring which is the localization of $R$ at the principal prime ideal $\pi R$. Under the above assumption, there exists an element $a \in R$ such that the height of $(\pi,a)$ is 2 and $R[\frac{1}{a}] \to S[\frac{1}{a}]$ is \'etale.

\begin{definition}[Basic ring]
\label{bigbasic}
Under the set-up as above, let us define
$$
R_{\infty}:=\bigcup_{n>0} R_n ,
$$
where $R_n:=V[[x_2^{p^{-n}},\ldots,x_d^{p^{-n}}]]$.
\end{definition}

The perfect closure of the integral domain $R/\pi R$ is isomorphic to $R_{\infty}/\pi R_{\infty}$. That is, the $p$-th power map on $R_{\infty}/\pi R_{\infty}$ is a bijection. The $\pi$-adic completion of $R_{\infty}$ is isomorphic to the ring of the ramified Witt vectors $\mathbf{W}_{\pi}(R_{\infty}/\pi R_{\infty})$, Namely, we have 
$$
R_{\infty} \hookrightarrow R_{\infty}^{\wedge} \cong \mathbf{W}_{\pi}(R_{\infty}/\pi R_{\infty}).
$$
By Proposition \ref{almostCM}, $R_{\infty}^{\wedge}$ is a normal domain. There is a commutative diagram:
$$
\begin{CD}
\mathbf{W}_{\pi}(R_{\infty}/\pi R_{\infty}) @>\mathbf{F}_{\pi}>\sim> \mathbf{W}_{\pi}(R_{\infty}/\pi R_{\infty}) \\
@V\Phi VV @V\Phi VV \\
R_{\infty}/\pi R_{\infty} @>\Frob^e>\sim> R_{\infty}/\pi R_{\infty} \\
\end{CD}
$$
in which $\Frob^e$ is the $q$-th power map and $\mathbf{F}_{\pi}(\pi)=\pi$ and $\Phi$ is the natural projection obtained by reduction modulo $\pi$. The $q$-Witt-Frobenius map $\mathbf{F}_{\pi}$ restricted to its sub-algebra $R_{\infty}$ defines a ring automorphism:
$$
\mathbf{F}_{\pi}:R_{\infty} \xrightarrow{\sim} R_{\infty};~\mathbf{F}(x_i)=x_i^q~(i=2,\ldots,d).
$$
It is easy to see that $R \to R_{\infty}$ is integral and faithfully flat.

\subsection{Maximal \'etale extensions}
First, we prove the following proposition.

\begin{proposition}
Let $A \hookrightarrow B$ be a torsion free ring extension such that $A$ is a normal domain and $B$ is reduced. Assume that $A \to C_1$ and $A \to C_2$ are finite \'etale extensions contained in $B$. Then there exists a finite \'etale extension $A \to C$ contained in $B$ such that $C$ contains both $C_1$ and $C_2$.
\end{proposition}

\begin{proof}
If $A \hookrightarrow C_1 \hookrightarrow B$ and $A \hookrightarrow C_2 \hookrightarrow B$ are finite \'etale sub-algebras, then the base change $A \to C_1 \otimes_A C_2$ is also finite \'etale. Using this fact, we prove that $A \hookrightarrow C_1 \cdot C_2 \hookrightarrow B$ is a finite \'etale sub-algebra over $A$. To see this, note first that $C_1 \otimes_A C_2$ is a normal ring by normality of $A$, since normality is preserved under \'etale extensions. Moreover, since $A$ is a domain, $C_1 \otimes_A C_2$ has finitely many minimal primes $P_1,\ldots.P_m$. So we have
$$
C_1 \otimes_A C_2 \cong \bigoplus_{i=1}^m (C_1 \otimes_A C_2)/P_i
$$
and $(C_1 \otimes_A C_2)/P_i$ is a normal domain. Let $J$ be the kernel of the map $g:C_1 \otimes_A C_2 \twoheadrightarrow C_1 \cdot C_2$ defined by $g(a_1 \otimes a_2)=a_1 a_2$. Then $J$ is a radical ideal and it is written as $\cap_{j \in \Lambda} Q_j$, where $Q_j$ range over all primes that are minimal over $J$ (the existence of a \textit{minimal} prime ideal in a ring is assured by Zorn's lemma). Moreover, $g(Q_j)$ is a minimal prime ideal of $C_1 \cdot C_2$. Since $A \to C_1 \cdot C_2$ is a torsion free module-finite extension and $A$ is a normal domain, it satisfies the Going-Down condition \cite[Theorem 2.2.7]{SwHu}. Then we see that the prime $g(Q_j)$ contracts to the zero ideal of $A$ in view of \cite[Lemma B.1.3]{SwHu} under the map $A \to C_1 \cdot C_2$. This shows that each $Q_j$ coincides with one of $P_1,\ldots,P_m$, from which we deduce that $C_1 \cdot C_2$ coincides with the localization $(C_1 \otimes_A C_2)[\frac{1}{e}]$ for some idempotent element $e \in C_1 \otimes_A C_2$. Hence
$$
A \to (C_1 \otimes_A C_2)[\frac{1}{e}] \cong C_1 \cdot C_2
$$
is finite \'etale and this establishes the proposition.
\end{proof}

Note that $C_1 \cdot C_2$ in the proposition is a finite direct product of normal domains. We make a definition of \textit{maximal \'etale extension} in the generality we need.

\begin{definition}
\label{maximaletale}
Let $A \hookrightarrow B$ be a torsion free ring extension such that $A$ is a normal domain and $B$ is reduced. Then we define the \textit{maximal \'etale extension} of $A$ contained in $B$ to be the filtered colimit of all finite \'etale $A$-algebras contained in $B$. We will write this extension as $A_B^{\mathrm{\acute{e}t}}$ or just as $A^{\mathrm{\acute{e}t}}$, if no confusion occurs.
\end{definition}

Note that the maximal \'etale extension of $A$ is a sub-algebra of $B$ and it is an integral ind-\'etale extension of $A$.

\begin{lemma}
The maximal \'etale extension $A^{\mathrm{\acute{e}t}}$ is a normal ring.
\end{lemma}

\begin{proof}
Note that the normality is preserved under \'etale extension and the filtered colimit of normal rings (with torsion free transition maps) is normal \cite[Proposition 19.3.1]{SwHu}. Thus, $A^{\mathrm{\acute{e}t}}$ is a normal ring by normality of $A$, that is, the localization of $A^{\mathrm{\acute{e}t}}$ at any of its maximal ideal is an integrally closed domain.
\end{proof}

\begin{proposition}
\label{maxetale}
Assume that $A$ is a normal domain and fix an embedding into the absolute integral closure $A \hookrightarrow A^+$ with $\overline{K}=\Frac(A^+)$ and let $A^{\mathrm{\acute{e}t}}$ be the maximal \'etale extension of $A$ contained in $A^+$. Then $K=\Frac(A) \to L=\Frac(A^{\mathrm{\acute{e}t}})$ is a Galois extension. In particular, we have $\sigma(A^{\mathrm{\acute{e}t}})=A^{\mathrm{\acute{e}t}}$ for any $\sigma \in \Hom_K(L,\overline{K})$.
\end{proposition}

\begin{proof}
The field extension $K \to L$ is ind-\'etale by construction, so it is separable. For normality, let $g \in \Hom_K(L,\overline{K})$. Then what we need to show is that $g(L)=L$. Since $g(A)=A$, $g(A^{\mathrm{\acute{e}t}})$ is an ind-\'etale extension over $A$ which is contained in $A^+$. Hence the composite ring $A^{\mathrm{\acute{e}t}} \cdot g(A^{\mathrm{\acute{e}t}})$ is ind-\'etale over $A$. Then by the maximal \'etale condition of $A^{\mathrm{\acute{e}t}}$ over $A$, we must have $g(A^{\mathrm{\acute{e}t}}) \subset A^{\mathrm{\acute{e}t}} \cdot g(A^{\mathrm{\acute{e}t}})=A^{\mathrm{\acute{e}t}}$. From this, it is easy to see that $g(A^{\mathrm{\acute{e}t}})=A^{\mathrm{\acute{e}t}}$. Thus $g(L)=L$ and $K \to L$ is a normal extension, as desired.
\end{proof}

We resume the notation as in the previous section. We note that $\pi R_{\infty}$ is a prime ideal of $R_{\infty}$ and the extension $R_{(\pi)} \to (R_{\infty})_{(\pi)}$ is an integral extension of discrete valuation rings. Moreover, the extension of residue fields:
$$
R_{(\pi)}/\pi R_{(\pi)} \to (R_{\infty})_{(\pi)}/\pi (R_{\infty})_{(\pi)}
$$
is purely inseparable and $(R_{\infty})_{(\pi)}/\pi (R_{\infty})_{(\pi)}$ is a perfect field which is identified with the perfect closure of $\Frac(\mathbb{F}_q[[x_2,\ldots,x_d]])$. Let $(R_{\infty})_{(\pi)}^+$ be the absolute integral closure of $(R_{\infty})_{(\pi)}$.

\begin{definition}
\label{crucialdef}
We define a normal domain $(R_{\infty})_{(\pi)}^{\mathrm{\acute{e}t}}$ to be the maximal \'etale extension of $(R_{\infty})_{(\pi)}$ contained in $(R_{\infty})_{(\pi)}^+$. To simplify the notation, we put
$$
\mathcal{R}:=(R_{\infty})_{(\pi)}^{\mathrm{\acute{e}t}}.
$$
Let us fix a maximal ideal $\mathcal{Q}$ of $\mathcal{R}$. We put
$$
\mathcal{K_Q}:=\mathcal{R}_{\mathcal{Q}}/\pi \mathcal{R}_{\mathcal{Q}},
$$
where $\mathcal{R}_{\mathcal{Q}}$ is the localization of $\mathcal{R}$ at $\mathcal{Q}$.
\end{definition}

\begin{remark}
Note that $\mathcal{R}$ is a normal domain of Krull dimension one, as it is integral over the discrete valuation ring $(R_{\infty})_{(\pi)}$. Moreover, $\mathcal{R}_{\mathcal{Q}}$ is the filtered colimit of discrete valuation rings $(\mathcal{R}_i,\pi_i)_{i \in \Lambda}$ such that $(\pi_i)=\pi \mathcal{R}_i$. From this, we see that $\mathcal{R}_{\mathcal{Q}}$ is a discrete valuation ring and $\mathcal{K_Q}$ is its residue field
\end{remark}

With the notation as in Definition \ref{crucialdef}, the ring extension
$$
(R_{\infty})_{(\pi)}/\pi(R_{\infty})_{(\pi)} \to \mathcal{R}/\pi \mathcal{R}
$$
is ind-\'etale. Since the class of perfect algebras is stable under taking ind-\'etale extensions, $\mathcal{R}/\pi\mathcal{R}$ is a perfect $V/\pi V$-algebra. Hence $\mathcal{K_Q}$ is a perfect field. Now let $\mathbf{W}_{\pi}(\mathcal{K_Q})$ denote its ring of ramified Witt vectors. Then $\mathbf{W}_{\pi}(\mathcal{K_Q})$ is an $R_{\infty}$-algebra.

\begin{lemma}
\label{torsionfreenormal}
$R_{\infty} \to \mathbf{W}_{\pi}(\mathcal{K_Q})$ is a torsion free ring extension and $\mathbf{W}_{\pi}(\mathcal{K_Q})$ is a complete discrete valuation ring.
\end{lemma}

\begin{proof}
For the first statement, $R_{\infty} \to \mathbf{W}_{\pi}(\mathcal{K_Q})$ factors as $R_{\infty} \to (R_{\infty})_{(\pi)} \to \mathbf{W}_{\pi}(\mathcal{K_Q})$ and $(R_{\infty})_{(\pi)}$ is a discrete valuation ring. The torsion freeness follows from this. The second statement is clear.
\end{proof}

\begin{lemma}
The $\pi$-adic completion of $\mathcal{R}_{\mathcal{Q}}$ is isomorphic to $\mathbf{W}_{\pi}(\mathcal{K_Q})$. Moreover, $\mathcal{R}_{\mathcal{Q}}$ is $\pi$-adically separated.
\end{lemma}

\begin{proof}
$\mathcal{R}_{\mathcal{Q}}$ is a $\pi$-torsion free ring and $\mathcal{K_Q}=\mathcal{R}_{\mathcal{Q}}/\pi \mathcal{R}_{\mathcal{Q}}$ is a perfect field which is isomorphic to $\mathbf{W}_{\pi}(\mathcal{K_Q})/\pi\mathbf{W}_{\pi}(\mathcal{K_Q})$. Hence the $\pi$-adic completion of $\mathcal{R}_{\mathcal{Q}}$ is canonically isomorphic to $\mathbf{W}_{\pi}(\mathcal{K_Q})$ in view of Proposition \ref{Serre2}. Since $\mathcal{R}_{\mathcal{Q}}$ is a discrete valuation ring, it is $\pi$-adically separated.
\end{proof}

By this lemma, there is a chain of ring injections:
\begin{equation}
\label{fundamental}
R_{\infty} \hookrightarrow \mathcal{R} _{\mathcal{Q}} \hookrightarrow \mathcal{R}_{\mathcal{Q}}^{\wedge} \cong \mathbf{W}_{\pi}(\mathcal{K_Q}),
\end{equation}
where $\mathcal{R}_{\mathcal{Q}}^{\wedge}$ denotes the $\pi$-adic completion of a ring $\mathcal{R}_{\mathcal{Q}}$. More precisely, $\mathcal{R}_{\mathcal{Q}} \hookrightarrow \mathbf{W}_{\pi}(\mathcal{K_Q})$ in $(\ref{fundamental})$ is defined as follows:

\begin{enumerate}
\item[$\bullet$]
$\mathcal{R}_{\mathcal{Q}} \hookrightarrow \mathcal{R}_{\mathcal{Q}}^{\wedge}$ is the canonical injection induced by the $\pi$-adic completion. The $V$-algebra isomorphism $\mathcal{R}_{\mathcal{Q}}^{\wedge} \cong \mathbf{W}_{\pi}(\mathcal{K_Q})$ is uniquely characterized by fixing an isomorphism on the residue fields modulo $\pi$. Namely, we have fixed an identity map on residue fields:
$$
\mathcal{R}_{\mathcal{Q}}^{\wedge}/\pi \mathcal{R}_{\mathcal{Q}}^{\wedge}=\mathcal{K_Q}= \mathbf{W}_{\pi}(\mathcal{K_Q})/\pi\mathbf{W}_{\pi}(\mathcal{K_Q}).
$$
\end{enumerate}

The sequence $(\ref{fundamental})$ will play an important role. On the other hand, since $R_{\infty}$ is $\pi$-adically separated, we have another commutative square:
$$
\begin{CD}
R_{\infty} @>\hookrightarrow>> R_{\infty}^{\wedge} \cong \mathbf{W}_{\pi}(R_{\infty}/\pi R_{\infty}) @>\hookrightarrow>> \mathbf{W}_{\pi}(\mathcal{K_Q}) \\
@V\mathbf{F}_{\pi}V\wr V @V\mathbf{F}_{\pi}V\wr V  @V\mathbf{F}_{\pi}V\wr V \\
R_{\infty} @>\hookrightarrow>> R_{\infty}^{\wedge} \cong \mathbf{W}_{\pi}(R_{\infty}/\pi R_{\infty}) @>\hookrightarrow>> \mathbf{W}_{\pi}(\mathcal{K_Q}) \\
\end{CD}
$$
where $\mathbf{W}_{\pi}(R_{\infty}/\pi R_{\infty}) \hookrightarrow \mathbf{W}_{\pi}(\mathcal{K_Q})$ is induced by an injection $R_{\infty}/\pi R_{\infty} \hookrightarrow \mathcal{K_Q}$.

\section{Witt-Frobenius stable algebras}
\label{sec7}

\subsection{Construction of some algebras}
Let the notation be as in Definition \ref{crucialdef}. In this section, we construct two big $R_{\infty}$-algebras as sub-algebras of $\mathbf{W}_{\pi}(\mathcal{K_Q})$. These algebras are defined as large integral extensions of $R$ contained in $\mathbf{W}_{\pi}(\mathcal{K_Q})$ such that they are stable under $\mathbf{F}_{\pi}$. However, it should be noted that if $B \subset \mathbf{W}_{\pi}(\mathcal{K_Q})$ is a sub-algebra, then it is generally not true that $\mathbf{F}_{\pi}(B) \subset B$ (see Example \ref{crucialexample} below). The failure of the stability of the $q$-Witt-Frobenius map forces us to take a certain huge integral extension of $R_{\infty}$ to achieve the stability. Let $\Ht I$ denote the height of an ideal $I$ in a ring $A$.

\begin{definition}
\label{basicring}
We fix an element $a \in R=V[[x_2,\ldots,x_d]]$ with the condition $\Ht(\pi,a)=2$ and define an $R_{\infty}$-algebra $T_n$ ($n=0,1,\ldots$) with the following conditions:
\begin{enumerate}
\item[$\bullet$]
$R_{\infty} \subset T_n \subset \mathbf{W}_{\pi}(\mathcal{K_Q})$ and $R_{\infty} \to T_n$ is an integral extension.

\item[$\bullet$]
The localization map
$$
R_{\infty}[\frac{1}{a_n}] \to T_n[\frac{1}{a_n}]
$$
is the maximal \'etale extension contained in $\mathbf{W}_{\pi}(\mathcal{K_Q})$ for $n \ge 0$, where we put
$$
a_n:=\prod_{k=-n}^n \mathbf{F}_{\pi}^k(a),
$$
where $\mathbf{F}_{\pi}^k$ is the $k$-th iterated $\mathbf{F}_{\pi}$ or $\mathbf{F}_{\pi}^{-1}$, depending on $k$ being $>0$ or $<0$, and $\mathbf{F}_{\pi}^0$ is the identity map (in particular, we have $a_0=a$). 
\end{enumerate}
Under the above notation, there is an increasing chain of rings:
$$
R_{\infty} \subset T_0 \subset T_1 \subset \cdots \subset \mathbf{W}_{\pi}(\mathcal{K_Q})
$$
and the filtered colimit $T_{\infty}:=\varinjlim_n T_n$ satisfies the inclusion $R_{\infty} \subset T_{\infty} \subset \mathbf{W}_{\pi}(\mathcal{K_Q})$.
\end{definition}

It is probably better to write $T_n^{\{a\}}$ rather than just $T_n$ to indicate that $T_n$ depends on $a \in R$, but we choose a simpler form to avoid the complication of symbols. Let us establish some properties of $T_{\infty}$.

\begin{lemma}
\label{lemma1}
Under the notation as above, we have the following statements:

\begin{enumerate}
\item[$\mathrm{(i)}$]
$T_n$ is a $\pi$-adically separated, normal domain for $n=0,1,\ldots,\infty$.

\item[$\mathrm{(ii)}$]
Assume that $S$ is a module-finite $R_n$-algebra such that 
$$
R_n[\frac{1}{a_m}] \to S[\frac{1}{a_m}]
$$ 
is finite \'etale for some $m \in \mathbf{N}$ and that $S$ is contained in $\mathbf{W}_{\pi}(\mathcal{K_Q})$. Then $S$ is contained in $T_{\infty}$. Moreover, $T_{\infty}$ does not depend on the choice of the maximal ideal $\mathcal{Q}$ of $\mathcal{R}$.

\item[$\mathrm{(iii)}$]
The ring extension 
$$
R_{\infty}/\pi R_{\infty}[\frac{1}{a}] \to T_{\infty}/\pi T_{\infty}[\frac{1}{a}]
$$ 
is ind-\'etale. In particular, $T_{\infty}/\pi T_{\infty}[\frac{1}{a}]$ is a perfect $V/\pi V$-algebra.

\item[$\mathrm{(iv)}$]
$(\pi,a)$ is a regular sequence on $T_{\infty}$. Moreover, $T_{\infty}/\pi T_{\infty}$ is a reduced $\mathbb{F}_p$-algebra.

\end{enumerate}
\end{lemma}

\begin{proof}
(i): Since $\mathbf{W}_{\pi}(\mathcal{K_Q})$ is a discrete valuation ring by Lemma \ref{torsionfreenormal}, it is clear that $T_n$ is $\pi$-adically separated and a normal domain.

(ii): Let us note the following facts:
\begin{enumerate}
\item[$\bullet$]
$R_{\infty}[\frac{1}{a_m}] \to T_m[\frac{1}{a_m}]$ is the maximal \'etale extension contained in $\mathbf{W}_{\pi}(\mathcal{K_Q})$. 

\item[$\bullet$]
The ring extension $R_{\infty}[\frac{1}{a_m}] \to (R_{\infty} \cdot S)[\frac{1}{a_m}]$ is \'etale and $(R_{\infty} \cdot S)[\frac{1}{a_m}] \subset \mathbf{W}_{\pi}(\mathcal{K_Q})$.
\end{enumerate}

From these facts, we see that $S$ is contained in $T_{\infty}$. Let us put $(T_{\infty})_{(\pi)}=T_{\infty} \otimes_{R_{\infty}} (R_{\infty})_{(\pi)}$ and $\mathcal{R}=(R_{\infty})_{(\pi)}^{\mathrm{\acute{e}t}}$ (see Definition \ref{crucialdef}). Then there is the following commutative diagram:
$$
\begin{CD}
R_{\infty} @>>> (R_{\infty})_{(\pi)}  @>>> \mathcal{R} \\
@VVV @VVV @VVV \\
T_{\infty} @>>> (T_{\infty})_{(\pi)}  @>>> \mathbf{W}_{\pi}(\mathcal{K_Q}) \\
\end{CD}
$$
in which every map is injective. Now we claim that $T_{\infty} \subset \mathcal{R}$. To prove the claim, let $\mathcal{R}^{\cl}$ be the integral closure of $\mathcal{R}$ in $\mathbf{W}_{\pi}(\mathcal{K_Q})$. Then we have $(T_{\infty})_{(\pi)} \subset \mathcal{R}^{\cl}$ by construction and the above diagram fits into the following commutative diagram:
$$
\begin{CD}
(R_{\infty})_{(\pi)} @>>> \mathcal{R} \\
@VVV @VVV \\
(T_{\infty})_{(\pi)} @>>> \mathcal{R}^{\cl} \\
\end{CD}
$$
Then, since $(R_{\infty})_{(\pi)} \to (T_{\infty})_{(\pi)}$ is an integral ind-\'etale extension, we have $(T_{\infty})_{(\pi)} \subset \mathcal{R}$ in view of the fact that
$(R_{\infty})_{(\pi)} \to \mathcal{R}$ is the maximal \'etale extension contained in $\mathcal{R}^{\cl}$. Hence we have $T_{\infty} \subset \mathcal{R}$ and this implies that $T_{\infty}$ does not depend on the choice of $\mathcal{Q}$.

(iii): For any $n \in \mathbf{Z}$, we have $\mathbf{F}_{\pi}^n(a) \equiv a^{q^n} \pmod{\pi R_{\infty}}$. By using this fact, we see that the horizontal map in the diagram:
$$
\begin{CD}
R_{\infty}/\pi R_{\infty}[\frac{1}{a}] @>>> T_n/\pi T_n[\frac{1}{a}] \\
@|@| \\
R_{\infty}/\pi R_{\infty}[\frac{1}{a_n}] @>>> T_n/\pi T_n[\frac{1}{a_n}]
\end{CD}
$$
is an ind-\'etale extension for all $n \ge 0$ by the \'etale base change. The statement follows from the fact that $T_{\infty}/\pi T_{\infty}[\frac{1}{a}]$ is the filtered colimit of all $T_n/\pi T_n[\frac{1}{a}]$.

(iv): Since $T_{\infty}$ is the filtered colimit of normal module-finite $R$-algebras, the assertion on the regularity of $(\pi,a)$ follows from Serre's normality criterion. By this result, the map $T_{\infty}/\pi T_{\infty} \to T_{\infty}/\pi T_{\infty}[\frac{1}{a}]$ is injective. Then by (iii), $T_{\infty}/\pi T_{\infty}[\frac{1}{a}]$ is reduced and hence its sub-algebra $T_{\infty}/\pi T_{\infty}$ is also reduced.
\end{proof}

We define another $R_{\infty}$-algebra which is larger than $T_{\infty}$.

\begin{definition}
We define $T^{\cl}_{\infty}$ to be the integral closure of $R_{\infty}$ in $\mathbf{W}_{\pi}(\mathcal{K_Q})$.
\end{definition}

We have a chain of inclusions $R_{\infty} \subset T_{\infty} \subset T^{\cl}_{\infty} \subset \mathbf{W}_{\pi}(\mathcal{K_Q})$. Although we do not use the ring $T^{\cl}_{\infty}$ in the present article, let us remark that it shares some properties with $T_{\infty}$.

\begin{lemma}
\label{lemma2}
Let $\mathbf{F}_{\pi}$ be the $q$-Witt-Frobenius map on $\mathbf{W}_{\pi}(\mathcal{K_Q})$. Then the restriction of $\mathbf{F}_{\pi}$ to its sub-algebra $T_{\infty}$ $($resp. $T^{\cl}_{\infty}$$)$ defines a ring automorphism, that is, $\mathbf{F}_{\pi}(T_{\infty})=T_{\infty}$ $($resp. $\mathbf{F}_{\pi}(T_{\infty}^{\cl})=T_{\infty}^{\cl}$$)$. 
\end{lemma}

\begin{proof}
Consider the commutative diagram:
$$
\begin{CD}
R_{\infty} @>>> T_{\infty} @>>> \mathbf{W}_{\pi}(\mathcal{K_Q}) \\
@V\mathbf{F}_{\pi}V\wr V @V\mathbf{F}_{\pi}V\wr V @V\mathbf{F}_{\pi}V\wr V \\
R_{\infty} @>>> \mathbf{F}_{\pi}(T_{\infty}) @>>> \mathbf{W}_{\pi}(\mathcal{K_Q}) \\
\end{CD}
$$
Then since $T_{\infty}$ is the colimit of all $T_n$, it suffices to show that $\mathbf{F}_{\pi}(T_n) \subset T_{\infty}$ and $\mathbf{F}_{\pi}^{-1}(T_n) \subset T_{\infty}$ for all $n>0$ to prove the lemma. By definition, the integral extension
$$
R_{\infty}[\frac{1}{a_n}] \to T_n[\frac{1}{a_n}]
$$
is ind-\'etale. Applying the Witt-Frobenius map, the horizontal maps in the following diagram:
$$
\begin{CD}
R_{\infty}[\frac{1}{a_n}] @>>> T_n[\frac{1}{a_n}] \\
@V\mathbf{F}_{\pi}V\wr V @V\mathbf{F}_{\pi}V\wr V \\
R_{\infty}[\frac{1}{\mathbf{F}_{\pi}(a_n)}] @>>> \mathbf{F}_{\pi}(T_n)[\frac{1}{\mathbf{F}_{\pi}(a_n)}] \\
\end{CD}
$$
are ind-\'etale. Since there is an equality: $\mathbf{F}_{\pi}(a_n) \cdot \mathbf{F}_{\pi}^{-n}(a) \cdot \mathbf{F}_{\pi}^{-(n+1)}(a)=a_{n+1}$, we have
$$
\begin{CD}
R_{\infty}[\frac{1}{\mathbf{F}_{\pi}(a_n)}][\frac{1}{\mathbf{F}_{\pi}^{-n}(a) \cdot \mathbf{F}_{\pi}^{-(n+1)}(a)}] @>>> \mathbf{F}_{\pi}(T_n)[\frac{1}{\mathbf{F}_{\pi}(a_n)}][\frac{1}{\mathbf{F}_{\pi}^{-n}(a) \cdot \mathbf{F}_{\pi}^{-(n+1)}(a)}] \\
@| @| \\
R_{\infty}[\frac{1}{a_{n+1}}] @>>> \mathbf{F}_{\pi}(T_n)[\frac{1}{a_{n+1}}] \\
\end{CD}
$$
So in view of the fact that $R_{\infty}[\frac{1}{a_{n+1}}] \to T_{n+1}[\frac{1}{a_{n+1}}]$ is the maximal \'etale extension contained in $\mathbf{W}_{\pi}(\mathcal{K_Q})$, we have
$$
\mathbf{F}_{\pi}(T_n) \subset T_{n+1},
$$
and hence $\mathbf{F}_{\pi}(T_n) \subset T_{\infty}$. The same type of reasoning using an equality: $\mathbf{F}^{-1}_{\pi}(a_n) \cdot \mathbf{F}_{\pi}^{n}(a) \cdot \mathbf{F}_{\pi}^{(n+1)}(a)=a_{n+1}$ yields the other inclusion: $\mathbf{F}_{\pi}^{-1}(T_n) \subset T_{\infty}$. 

By a similar argument, we can establish an equality: $\mathbf{F}_{\pi}(T_{\infty}^{\cl})=T_{\infty}^{\cl}$ by utilizing the following commutative diagram:
$$
\begin{CD}
R_{\infty} @>>> T_{\infty}^{\cl} @>>> \mathbf{W}_{\pi}(\mathcal{K_Q}) \\
@V\mathbf{F}_{\pi}V\wr V @V\mathbf{F}_{\pi}V\wr V @V\mathbf{F}_{\pi}V\wr V \\
R_{\infty} @>>> \mathbf{F}(T_{\infty}^{\cl}) @>>> \mathbf{W}_{\pi}(\mathcal{K_Q}) \\
\end{CD}
$$
Indeed, we see from this diagram that $R_{\infty} \to \mathbf{F}_{\pi}(T_{\infty}^{\cl})$ is an integral extension, so that we have $\mathbf{F}_{\pi}(T_{\infty}^{\cl}) \subset T_{\infty}^{\cl}$ and likewise, $\mathbf{F}_{\pi}^{-1}(T_{\infty}^{\cl}) \subset T_{\infty}^{\cl}$.
\end{proof}

\begin{example}
\label{crucialexample}
This example gives the failure of the stability of the Witt-Frobenius map. We set
$$
R=\mathbf{W}(k)[[x]]
$$
for a perfect field $k$ of characteristic $p>2$, 
$$
R_{\infty}=\bigcup_{n>0} \mathbf{W}(k)[[x^{p^{-n}}]]
$$
and a normal domain: 
$$
\mathcal{S}=R_{\infty}[\sqrt{p+x}] \subset R^+.
$$ 
Then $R_{\infty} \to \mathcal{S}$ is a finite extension of degree 2 and the induced map $\mathcal{K_Q}=(R_{\infty})_{(p)}/p(R_{\infty})_{(p)} \to \mathcal{L}=\mathcal{S}_{(p)}/p\mathcal{S}_{(p)}$ is a finite extension of perfect fields. After taking the ring of Witt vectors, we get a commutative diagram:
$$
\begin{CD}
\mathcal{S} @>\hookrightarrow>> \mathcal{S}_{(p)} @>\hookrightarrow>> \mathcal{S}_{(p)}^{\wedge} \cong \mathbf{W}(\mathcal{L}) \\
@AAA @AAA @AAA \\
R_{\infty} @>\hookrightarrow>> (R_{\infty})_{(p)} @>\hookrightarrow>> (R_{\infty})_{(p)}^{\wedge} \cong \mathbf{W}(\mathcal{K_Q}) \\
\end{CD}
$$
Let $\mathbf{F}$ be the $p$-Witt-Frobenius map on $\mathcal{S}_{(p)}^{\wedge}$. Since we have $\mathbf{F}(p+x)=p+x^p$, it follows that 
$$
\mathbf{F}(\sqrt{p+x})=\pm\sqrt{p+x^p} \in \mathcal{S}_{(p)}^{\wedge}
$$
and hence $\mathbf{F}(\sqrt{p+x}) \notin \mathcal{S}$. The failure of the stability of $\mathbf{F}$ seems to be a major obstacle for constructing a ring of mixed characteristic with a perfect quotient modulo a non-zero divisor.
\end{example}

\section{Main theorem}
\label{sec8}
We will prove the main theorem.

\begin{theorem}
\label{theorem1}
Let $S$ be a complete local domain of mixed characteristic $p>0$ with finite residue field. Then there exists an $S$-algebra $T$ with a non-zero non-unit element $\pi \in T$ such that the following conditions hold:

\begin{enumerate}
\item[$\mathrm{(i)}$]
$T$ is a normal domain and $S \subset T \subset S^+$. 

\item[$\mathrm{(ii)}$]
$T/\pi T$ is a reduced $\mathbb{F}_p$-algebra.

\item[$\mathrm{(iii)}$]
For any prime ideal $P$ of $T$ that is minimal over $\pi T$, the Frobenius endomorphism is bijective on the quotient ring $T/P$.
\end{enumerate}
\end{theorem}

\begin{proof}
After replacing $S$ with a larger module-finite domain $S'$, there exists a torsion free module-finite map 
$$
R=V[[t_2,\ldots,t_d]] \to S'
$$ 
as stated in Theorem \ref{gabber} (Cohen-Gabber theorem). More concretely, $S'$ is a local normal domain and $R[\frac{1}{a}] \to S'[\frac{1}{a}]$ is \'etale for some $a \in R$ with $\Ht(\pi,a)=2$. Then we can construct an $R_{\infty}$-algebra $T_{\infty}$ associated to the module-finite extension $R \to S'$, where $T_{\infty}$ is as in Definition \ref{basicring}. That is, we have $S \subset S' \subset T_{\infty}$. Hence it suffices to construct an $S'$-algebra $T$, as required in the theorem. By replacing $S$ with $S'$, we may assume that $S$ fits into the set up of Cohen-Gabber theorem. Let $q^e=|\mathbb{F}|$ with $\mathbb{F}=V/\pi V$. 

Under the notation as above, we prove that $T:=T_{\infty}$ satisfies all requirements in the theorem. As to (i) and (ii), one just applies Lemma \ref{lemma1}. So it remains to prove (iii). Let $\mathbf{F}_{\pi}$ be the $q$-Witt-Frobenius map on $\mathbf{W}_{\pi}(\mathcal{K_Q})$. According to Lemma \ref{lemma2}, we have $\mathbf{F}_{\pi}(T_{\infty})=T_{\infty}$ and we have the following commutative diagram:
$$
\begin{CD}
T_{\infty} @>\hookrightarrow>> \mathbf{W}_{\pi}(\mathcal{K_Q}) @>>> \mathcal{K_Q} \\
@V\mathbf{F}_{\pi}V\wr V @V\mathbf{F}_{\pi}V\wr V @V\Frob^eV\wr V \\
T_{\infty} @>\hookrightarrow>> \mathbf{W}_{\pi}(\mathcal{K_Q}) @>>> \mathcal{K_Q} \\
\end{CD}
$$
Let $P:=T_{\infty} \cap \pi\mathbf{W}_{\pi}(\mathcal{K_Q})$. Then $P$ is a prime ideal that is the kernel of the composite ring map $T_{\infty} \to \mathbf{W}_{\pi}(\mathcal{K_Q}) \to \mathcal{K_Q}$. Moreover, we have $\mathbf{F}_{\pi}(b) \in T_{\infty}$ and $\mathbf{F}_{\pi}(b)-b^q \in P$ for any $b \in T_{\infty}$. By the commutativity of the above diagram, it follows that $\mathbf{F}_{\pi}(P)=P$. Hence we have the commutative diagram:
$$
\begin{CD}
T_{\infty}/P @>\hookrightarrow>> \mathcal{K_Q} \\
@V\overline{\mathbf{F}}_{\pi}V\wr V @V\Frob^eV\wr V \\
T_{\infty}/P @>\hookrightarrow>> \mathcal{K_Q} \\
\end{CD}
$$
where $\overline{\mathbf{F}}_{\pi}$ is the $q$-th power map. Since the $q$-th power map is bijective on $T_{\infty}/P$, we see that $T_{\infty}/P$ is a perfect $\mathbb{F}_p$-algebra.

Next we prove that $P$ is minimal over $\pi T_{\infty}$. Recall that there is a ring map $R_{\infty} \to T_{\infty} \to \mathbf{W}_{\pi}(\mathcal{K_Q})$ and it gives $\pi R_{\infty}=R_{\infty} \cap \pi \mathbf{W}_{\pi}(\mathcal{K_Q})$, because there is a chain of injections:
$$
R_{\infty}/\pi R_{\infty} \hookrightarrow (R_{\infty})_{(\pi)}/\pi (R_{\infty})_{(\pi)} \hookrightarrow \mathcal{K_Q}=\mathcal{R}_{\mathcal{Q}}/\pi \mathcal{R}_{\mathcal{Q}},
$$
where the second map is a field extension, $\mathcal{R}:=(R_{\infty})_{(\pi)}^{\mathrm{\acute{e}t}}$ is the maximal \'etale extension of the discrete valuation ring $(R_{\infty})_{(\pi)}$, and $\mathcal{Q}$ is a maximal ideal of $\mathcal{R}$. We have $\pi R_{\infty}=R_{\infty} \cap \pi \mathbf{W}_{\pi}(\mathcal{K_Q})$, $P=T_{\infty} \cap \pi\mathbf{W}_{\pi}(\mathcal{K_Q})$ and $R_{\infty} \hookrightarrow T_{\infty}$ is an integral extension of domains. Thus, it follows that $P$ is a minimal prime over $\pi T_{\infty}$. 

Finally, we prove that $P$ may be taken to be any prime ideal that is minimal over $\pi T_{\infty}$. By Lemma \ref{lemma1}, we have $T_{\infty} \subset \mathcal{R}$. Note that $(T_{\infty})_{(\pi)} \to \mathcal{R}$ is an integral extension of normal domains of Krull dimension one. Then the set of prime ideals of $T_{\infty}$ that are minimal over $\pi T_{\infty}$ corresponds bijectively with the set of maximal ideals of $(T_{\infty})_{(\pi)}$. Any maximal ideal of $(T_{\infty})_{(\pi)}$ is obtained as the pull back of some maximal ideal of $\mathcal{R}$ by Lying-Over Theorem \cite[Theorem 2.2.2]{SwHu}. Therefore, for any such prime ideal $P \subset T_{\infty}$, one can find a maximal ideal $\mathcal{Q}$ of $\mathcal{R}$ such that $P=T_{\infty} \cap \mathcal{Q}$ under the composite map $T_{\infty} \to (T_{\infty})_{(\pi)} \to \mathcal{R}$. Since the construction of $\mathbf{W}_{\pi}(\mathcal{K_Q})$ is valid for any maximal ideal $\mathcal{Q}$ of $\mathcal{R}$, it follows that $P$ can be chosen to be an arbitrary prime ideal that is minimal over $\pi T_{\infty}$. Hence $T:=T_{\infty}$ has all desired properties.
\end{proof}

Fix a finite field $\mathbb{F}$ of characteristic $p>0$ and let $\mathbf{W}(\mathbb{F}) \to (V,\pi,\mathbb{F})$ be a totally ramified extension of discrete valuation rings. Then we have the following corollary.

\begin{corollary}
\label{theorem2}
Assume that $R:=V[[x_2,\ldots,x_d]] \to S$ is a module-finite extension of complete local domains such that $R[\frac{1}{a}] \to S[\frac{1}{a}]$ is \'etale for some $a \in R$ and the height of the ideal $(\pi,a)$ of $R$ is 2. Then the $S$-algebra $T$ in Theorem \ref{theorem1} can be taken to satisfy the following properties:
\begin{enumerate}
\item[$\mathrm{(i)}$]
There is a commutative diagram of integral domains
$$
\begin{CD}
R_{\infty} @>>> T \\
@AAA @AAA \\
R @>>> S \\
\end{CD}
$$
where each map is injective and integral. Moreover, the ring map
$$
R_{\infty}/\pi R_{\infty}[\frac{1}{a}] \to T/\pi T[\frac{1}{a}],
$$
which is induced by $R_{\infty} \to T$, is the filtered colimit of finite \'etale $R_{\infty}/\pi R_{\infty}[\frac{1}{a}]$-algebras and the Frobenius endomorphism is bijective on $T/\pi T[\frac{1}{a}]$.

\item[$\mathrm{(ii)}$]
Fix a prime ideal $P$ of $T$ that is minimal over $\pi T$. Then there exists a ring automorphism:
$$
\mathbf{F}:T \xrightarrow{\sim} T
$$
such that $\mathbf{F}(P)=P$ and the induced map $\overline{\mathbf{F}}:T/P \xrightarrow{\sim} T/P$ coincides with the $q$-th power map with $q:=|\mathbb{F}|$ and $\mathbb{F}=V/\pi V$.
\end{enumerate}
\end{corollary}

\begin{proof}
Let us take $T:=T_{\infty}$ as in Theorem \ref{theorem1}. Then the first statement is due to Lemma \ref{lemma1}. For the second statement, the desired ring automorphism $\mathbf{F}:T \xrightarrow{\sim} T$ is given as the restriction of the $q$-Witt-Frobenius map $\mathbf{F}_{\pi}$ from $\mathbf{W}_{\pi}(\mathcal{K_Q})$ to the sub-algebra $T_{\infty}$. Indeed, the corollary follows from Theorem \ref{theorem1}.
\end{proof}

It will be worth investigating the following question.

\begin{question}
Let the $R_{\infty}$-algbera $T$ be as in Theorem \ref{theorem1}. Then it is true that $T/\pi T$ is a perfect $\mathbb{F}_p$-algebra?
\end{question}

\begin{remark}
\begin{enumerate}
\item
Let us consider the special case where $R=V[[x_2,\ldots,x_d]] \to S$ is ramified along the normal crossing divisor of $\Spec R$. Put
$$
a=a_{\Lambda}:=\prod_{i \in \Lambda} x_i \in R
$$
for a non-empty subset $\Lambda \subset \{2,\ldots,d\}$ in Definition \ref{basicring}. Assume that $R[\frac{1}{pa}] \to S[\frac{1}{pa}]$ is \'etale. Then the following result holds: Let us define
$$
R_n:=V[[x_2^{\frac{1}{n!}},\ldots,x_d^{\frac{1}{n!}}]]
$$
and let $S_n$ be the normalization of $(R_n \otimes_R S)_{\rm{red}}$ in its total ring of fractions. Then $R_n[\frac{1}{p}] \to S_n[\frac{1}{p}]$ is \'etale for $n \gg 0$ by Abhyankar's lemma \cite[XIII, Prop 5.2, Cor 5.3]{GroRay}.

\item
$T_{\infty}$ is a strictly henselian quasi-local normal domain. Indeed, since $R=V[[x_2,\ldots,x_d]]$ is a complete local domain and $T_{\infty}$ is its integral extension domain, $T_{\infty}$ is a henselian quasi-local domain. Moreover, let $\mathbb{F} \to \mathbb{F}'$ be a finite field extension. Then there exists a finite \'etale extension $V \to W$ of complete discrete valuations rings whose residue field extension is $\mathbb{F} \to \mathbb{F}'$. Hence $R \to W[[x_2,\ldots,x_d]]$ is finite \'etale and $W[[x_2,\ldots,x_d]] \subset T_{\infty}$. In other words, the residue field of $T_{\infty}$ is separably (algebraically) closed.

\item
$\Frac(R_{\infty}) \to \Frac(T_{\infty})$ is a Galois extension, which can be checked in view of the definition of $\mathcal{R}$ and Lemma \ref{maxetale}. However, it is not true that $\Frac(R_{\infty}) \to \Frac(T_{\infty}^{\cl})$ is Galois. This is already apparent in the construction of the henselization of the discrete valuation ring $V$ as a sub-algebra of the completion $\widehat{V}$ \cite[Lemma 6.2.5]{GR} and a reference given in its proof.
\end{enumerate}
\end{remark}

\section{Construction of semiperfect algebras}
\label{sec9}

The main result in this section is obtained by allowing a deep ramification over $p$ for a ring with mixed characteristic $p>0$, which enables us to prove the surjectivity of the Frobenius map. Indeed, we construct an algebra which contains roots of the equation $X^n-p=0$ for any $n \in \mathbb{N}$. It has applications to the construction of big Cohen-Macaulay algebras via Fontaine rings (see \cite{Rob} and \cite{Shim1} and Remark \ref{Fontainering} below). We first prove a crucial lemma.

\begin{lemma}
\label{semiperfect}
Let $A$ be a $p$-torsion free ring such that $A/pA \ne 0$ for a prime integer $p>0$. Assume that $A$ is either a $p$-adically complete normal domain, or a henselian quasi-local normal domain. Define $\overline{A}$ to be a unique $A$-algebra such that $A \subset \overline{A} \subset A^+$ and the localization map:
$$
A[\frac{1}{p}] \to \overline{A}[\frac{1}{p}]
$$
is the maximal \'etale extension. Then there exists an element $\pi \in \overline{A}$ such that $\pi^p=p$. Moreover, $\overline{A}$ is a normal domain, the Frobenius endomorphism is surjective on $\overline{A}/p\overline{A}$ and there is a ring isomorphism:
$$
\overline{A}/\pi \overline{A} \cong \overline{A}/p\overline{A}
$$
which is defined by $x \pmod{\pi \overline{A}} \mapsto x^p \pmod{p\overline{A}}$.
\end{lemma}

\begin{proof}
Since $A$ is a normal domain and
$$
A[\frac{1}{p}] \to \overline{A}[\frac{1}{p}]
$$
is ind-\'etale, $\overline{A}[\frac{1}{p}]$ is normal. Hence $\overline{A}$ is a normal domain by maximality.

Note that $pA \ne A$ implies that $p$ is not a unit element of $A$. If $A$ is $p$-adically complete, then $p$ is contained in the Jacobson radical of $A$ and so in that of $\overline{A}$. Next, if $A$ is a henselian quasi-local domain, then $p$ is contained in the unique maximal ideal of $A$ and thus in the unique maximal ideal of $\overline{A}$. Pick an element $b \in \overline{A}$ and consider a polynomial
$$
f(X):=X^{p^2}-pX-b \in \overline{A}[X].
$$
Then $f'(X)=p^2X^{p^2-1}-p=p(pX^{p^2-1}-1)$ and the localization map:
$$
\overline{A}[\frac{1}{p}] \to \overline{A}[X]/(f(X))[\frac{1}{p}]
$$
is finite \'etale, because $p$ is contained in the Jacobson radical of $\overline{A}[X]/(f(X))$ and therefore, the image of $pX^{p^2-1}-1$ in $\overline{A}[X]/(f(X))$ is a unit element. There exist an element $a \in A^+$ such that $f(a)=0$ and a commutative diagram:
$$
\begin{CD}
\overline{A} @>>>  \overline{A}[X]/(f(X)) \\
@| @VVV \\
\overline{A} @>>>  \overline{A}[a] @>>> A^+ \\
\end{CD}
$$
where $\overline{A}[X]/(f(X)) \to \overline{A}[a]$ is defined by mapping $X$ to $a$. Then $\overline{A}[X]/(f(X))][\frac{1}{p}]$ is isomorphic to a finite product of normal domains in view of the normality of $\overline{A}$, and $\overline{A}[a][\frac{1}{p}]$ is isomorphic to one of the factors of $\overline{A}[X]/(f(X))][\frac{1}{p}]$. This shows that 
$$
A[\frac{1}{p}] \to  \overline{A}[a][\frac{1}{p}]
$$
is ind-\'etale. Since $A[\frac{1}{p}] \to \overline{A}[\frac{1}{p}]$ is the maximal \'etale extension contained in $A^+[\frac{1}{p}]$, it follows that $a \in \overline{A}$. Finally, we have 
$$
a^{p^2}-b \equiv a^{p^2}-pa-b \equiv 0 \pmod{p \overline{A}}
$$ 
and $(a^p)^p \equiv b \pmod{p \overline{A}}$. This proves that the Frobenius endomorphism is surjective on $\overline{A}/p\overline{A}$. 

Let $\pi \in A^+$ be a root of the equation $X^p-p=0$. Since $\overline{A} \to \overline{A}[\pi]$ is \'etale after inverting $p$, we have $\pi \in \overline{A}$. To deduce an isomorphism $\overline{A}/\pi \overline{A} \cong \overline{A}/p\overline{A}$, it suffices to show that the kernel of the Frobenius endomorphism:
$$
\Frob: \overline{A}/p\overline{A} \to \overline{A}/p\overline{A}
$$
is principally generated by $\pi$. Assume that $\overline{x}^p=0$ for $\overline{x} \in \overline{A}/p \overline{A}$ with its lift $x \in \overline{A}$. Then we can write $x^p=p \cdot b$ for some $b \in \overline{A}$, which implies that $x=\pi \cdot b'$ with $b' \in A^+$ and
$$
b' \in \overline{A}[\frac{1}{\pi}] \cap A^+.
$$
Since $\overline{A}$ is integrally closed in the field of fractions, we have $b' \in \overline{A}$ and $x \in \pi \overline{A}$. This finishes the proof of the lemma.
\end{proof}

Now we prove the following theorem.

\begin{theorem}
\label{theorem3}
Let $S$ be a complete local domain with mixed characteristic $p>0$. Then there exists an $S$-algebra $T$ such that the following hold:

\begin{enumerate}
\item[$\mathrm{(i)}$]
$T$ is a normal domain and $S \subset T \subset S^+$.

\item[$\mathrm{(ii)}$]
There is an element $\pi \in T$ such that $\pi^p=p$ and the Frobenius endomorphism is surjective on $T/pT$, which induces an isomorphism:
$$
T/\pi T \cong T/pT.
$$

\item[$\mathrm{(iii)}$]
There exist a complete discrete valuation ring $V$, a regular local sub-algebra
$$
R:=V[[t_2,\ldots,t_d]] \subset T
$$
together with an element $a \in R$, and a complete local normal domain $S'$ such that $R \subset S' \subset T$, where $R \to S'$ is module-finite, $S' \to T$ is integral, the height of the ideal $(p,a)$ of $R$ is 2, and the localization maps:
$$
R[\frac{1}{a}] \to S'[\frac{1}{a}]~\mbox{and}~S'[\frac{1}{p}] \to T[\frac{1}{p}]
$$
are ind-\'etale. In particular, 
$$
R[\frac{1}{pa}] \to T[\frac{1}{pa}]
$$
is ind-\'etale.
\end{enumerate}
\end{theorem}

\begin{proof}
Since $S$ is a complete local domain by assumption, its module-finite extension normal domain satisfies the hypothesis of Lemma \ref{semiperfect}. By Theorem \ref{gabber}, there exists a module-finite extension $S \to S'$, a complete discrete valuation ring $V$ and a module-finite extension $R:=V[[t_2,\ldots,t_d]] \to S'$ such that $S'$ is normal and
$$
R[\frac{1}{a}] \to S'[\frac{1}{a}]
$$ 
is \'etale, where $a \in R$ satisfies the condition $\Ht(p,a)=2$. We define $T$ to be a normal domain such that $S' \subset T \subset S'^+$ and the localization map:
$$
S'[\frac{1}{p}] \to T[\frac{1}{p}]
$$
is the maximal \'etale extension in $S'^+[\frac{1}{p}]$. It follows that $T/pT$ is a semiperfect $\mathbb{F}_p$-algebra having all required properties in view of Lemma \ref{semiperfect}. This completes the proof of the theorem.
\end{proof}

\begin{remark}
\label{Fontainering}
Let us briefly recall the definition of \textit{Fontaine rings} (see \cite{Rob} and \cite{Shim1} for details). Let $A$ be a ring with $A/pA \ne 0$ for a fixed prime $p>0$. The Fontiane ring of $A$ is defined as the projective limit:
$$
\mathbf{E}(A):=\varprojlim_{n \in \mathbb{N}} A_n,
$$
where we put $A_n=A/pA$ and the transition map $A_{n+1} \to A_n$ is the $p$-th power map. Explicitly, an element of $\mathbf{E}(A)$ is written as $(a_0,a_1,\ldots,a_n,\ldots)$ such that $a_i \in A/pA$ and $a_{i+1}^p=a_i$. It is easy to see from the definition that the Fontaine ring is a perfect $\mathbb{F}_p$-algebra. In \cite{Shim1}, we considered the Fontaine ring in the case $A=R^+$ for a complete local domain $R$. In the articles \cite{An1}, \cite{An2}, Andr\'e proved that every complete local domain of mixed characteristic maps to an integral almost perfectoid almost Cohen-Macaulay algebra. For notation: The definition of an \textit{almost Cohen-Macaulay algebra} is the same as in \cite{Rob}. We say that a ring $A$ is \textit{integral almost perfectoid}, if it is $p$-adically complete, flat and the Frobenius endomorphism on $A/pA$ is almost surjective and its kernel is equal to $p^{\frac{1}{p}}A$. We ask the following question.

\begin{question}
Does every complete local domain of mixed characteristic possess an integral perfectoid big Cohen-Macaulay algebra?
Does every complete local domain of mixed characteristic admit sufficiently many big Cohen-Macaulay algebras?
\end{question}

It seems that the existence of such an algebra is useful in constructing a certain closure operation of ideals of Noetherian rings in mixed characteristic (an analogue of tight closure by Hochster and Huneke). By taking $A=T$ as constructed in Theorem \ref{theorem3} combined with the techniques developed in \cite{Shim1} and Andr\'e's results, we plan to study this question in \cite{Shim4}.
\end{remark}

\end{document}